\documentclass[12pt]{amsart}
\usepackage{amsmath}
\usepackage{amsfonts}
\usepackage{amssymb}
\usepackage{graphicx}

\newcommand{\dom}{\text{dom }}

\newtheorem{theorem}{Theorem}
\newtheorem{question}{Question}

\newtheorem{example}[theorem]{Example}

\newtheorem{lemma}[theorem]{Lemma}
\newtheorem{notation}[theorem]{Notation}

\newtheorem{proposition}[theorem]{Proposition}

\DeclareMathOperator{\supp}{supp}

\begin{document}

\title[Small powers of countably compact groups]{A van Douwen-like ZFC theorem for small powers of countably compact groups without non-trivial convergent sequences }

\author[A. H. Tomita]{Artur Hideyuki Tomita}
\address{Instituto de Matem\'atica e Estat\'istica, Universidade de S\~ao Paulo \\ Rua do Mat\~ao, 1010, CEP 05508-090, S\~ao Paulo, Brazil}
\email{tomita@ime.usp.br}


\dedicatory{To the Memory of Professor W. W. Comfort }

\date{}

\footnote{The author has received support from FAPESP Aux\' \i lio regular de pesquisa 2012/01490-9 and CNPq Produtividade em Pesquisa 305612/2010-7 during the research that led to this work. The author received support from CNPq Produtividade em Pesquisa 307130/2013-4 and CNPq Projeto Universal N. 483734/2013-6 Universal during the preparation and revision of the manuscript. Final revision under support from FAPESP 2016/26216-8.}

\begin{abstract}

We show that if $\kappa \leq \omega$ and there exists a group topology without non-trivial convergent sequences on an  Abelian group $H$ such that $H^n$ is countably compact for each $n<\kappa$ then there exists a topological group $G$ such that $G^n$ is countably compact for each $n <\kappa$ and $G^{\kappa}$ is not countably compact. If  in addition $H$ is torsion, then the result above holds for $\kappa=\omega_1$. Combining with other results in the literature, we show that:

 $a)$ Assuming ${\mathfrak c}$ incomparable selective ultrafilters, for each $n \in \omega$, there exists a group topology on the free Abelian group $G$ such that $G^n$ is countably compact and $G^{n+1}$ is not countably compact. (It was already know for $\omega$).

 $b)$ If  $\kappa \in \omega \cup \{\omega\} \cup \{\omega_1\}$, there exists in ZFC a topological group $G$ such that $G^\gamma$ is countably compact for each cardinal $\gamma <\kappa$ and $G^\kappa$ is not countably compact.
\end{abstract}



\maketitle

\section{Introduction}

\subsection{Countably compact groups without non-trivial convergent sequences and van Douwen's theorem}

In 1980, van Douwen \cite{vD80} showed in ZFC that if there exists a countably compact group without non-trivial convergent sequences that is a subgroup of $2^{\mathfrak c}$, then there exist two countably compact subgroups  whose product is not countably compact. In the same paper, van Douwen produced  a countably compact subgroup
of $2^{\mathfrak c}$, without non-trivial convergent sequences from Martin's Axiom (a CH example was earlier produced by H\' ajnal and Juh\' asz \cite{HaJu76}). Combining his two results, it follows that countable compactness in the class of topological groups is not productive when Martin's Axiom holds.

In 1991, Hart and van Mill \cite{HvM91}
produced from Martin's Axiom for countable posets the first countably compact group whose square is not countably compact. The approach differs from van Douwen's as their example contains many
convergent sequences.

 In 2004, Tomita and Watson \cite{ToWa04} constructed $p$-compact groups (in particular countably compact groups) from a selective ultrafilter $p$. With this technique, Garcia-Ferreira, Tomita and Watson \cite{GaToWa05} obtained a countably compact group without non-trivial convergent sequences from a single selective ultrafilter. 
 
 Szeptycki and Tomita \cite{SzTo09} showed that in the Random model there exists a countably compact group without non-trivial convergent sequences, giving  the first example that does not depend on selective ultrafilters. In \cite{SzTo?}, it was showed that the countable power of this example is countably compact.
 
 Recently, Hru\v sak, van Mill, Ramos and Shelah showed in ZFC that there exists a Boolean  group without non-trivial convergent sequences, answering a question of van Douwen from 1980. Using van Douwen's theorem, they also answer in ZFC a 1966 question of Comfort about the non-productivity of countable compactness in the class of topological groups. Hru\v sak has informed that the construction can be easily modified to obtain $G^\omega$ countably compact.

In 2005, Tomita \cite{To05} improved van Douwen's theorem showing that if there exists a countably compact Abelian group without non-trivial convergent sequences then there exists a countably compact Abelian group whose square is not countably compact.

Tkachenko \cite{Tk90} showed that under CH, there exists a countably compact free Abelian group.  Koszmider, Tomita and Watson \cite{KoToWa00} obtained one from Martin's Axiom for countable posets and Madariaga Garcia and Tomita \cite{MaTo07} obtained an example from the existence of ${\mathfrak c}$ selective ultrafilters.

\subsection{Powers of countably compact groups and Comfort's Question}

 Inspired by the example of Hart and van Mill and a result of Ginsburg and Saks \cite{GiSa75}, Comfort asked the following question in the Open Problems in Topology  \cite{Co90}:

\begin{question} (Question 477, Open Problems in Topology, 1990) Is there, for every (not necessarily infinite) cardinal number
$\kappa \leq 2^{\mathfrak c}$, a topological group $G$ such that $G^\gamma$ is countably compact for all
cardinals $\gamma < \kappa$, but $G^\kappa$ is not countably compact?
\end{question}

The result in \cite{GiSa75} implies that the question above only makes sense for cardinals not greater than $2^{\mathfrak c}$.

Partial results for some finite cardinals other than $2$ were obtained in \cite{To96} and \cite{To99} using Martin's Axiom for countable posets. Under the same axiom, finite cardinals were solved in \cite{To05ta}.

In \cite{To05fm} the question was answered consistently. It was showed under some cardinal restriction and assuming the existence of $2^{\mathfrak c}$ incomparable selective ultrafilters, there exists a topological group as in Comfort's question for every cardinal $\leq 2^{\mathfrak c}$.  The examples obtained are subgroups of a product of copies of $2$.

Sanchis and Tomita \cite{SaTo12} showed that if there exists a selective ultrafilter then there exists a topological group of order $2$ without non-trivial convergent sequences as in Comfort's question, for each cardinal $\alpha \leq \omega_1$.

In \cite{KoToWa00} and \cite{MaTo07}, it was showed that there exist  a countably compact group topology on the free Abelian group of cardinality ${\mathfrak c}$ whose square is not countably compact from Martin's Axiom and the existence of ${\mathfrak c}$ selective ultrafilters, respectively.

 Tomita \cite{To?}  showed that, assuming the existence of ${\mathfrak c}$ incomparable selective ultrafilters, there exists a topological free Abelian group without non-trivial convergent sequences whose finite powers are countably compact, improving the square result by Boero and Tomita \cite{BoTo11}.
 
 \subsection{A van Douwen like theorem for Comfort's Question}

In this work we obtain a van Douwen like theorem for  Abelian groups without non-trivial convergent sequences
whose small powers are countably compact.

We show that given a cardinal $\kappa \leq \omega_1$ and there exists a topological Abelian group of finite order such that $H^\gamma$ is countably compact for each cardinal $\gamma<\kappa$, then there exists a topological group $G$ as in Comfort's question $477$ for $\kappa$. 

Applying the recent result of Hru\v sak, van Mill, Ramos and Shelah, it follows that Comfort's Question $477$ is settled in ZFC for every cardinal $\leq \omega_1$.

We also show that if $\kappa \leq \omega$ and there exists a non-torsion topological Abelian group such that $G^\gamma$ is countably compact for each $\gamma<\kappa$, then there exists a topological free Abelian group as in Comfort's question for $\kappa$. Applying this result in the example in \cite{To?}, it follows  from the existence of ${\mathfrak c}$ selective ultrafilters that there exists, for each finite cardinal $M$, a topological free Abelian group $G$ such that $G^M$ is countably compact, but $G^{M+1}$ is not.

\section{Countably compact non torsion Abelian groups}

\subsection{Family of sequences that suffice to keep countable compactness in small powers}

We start defining the families that will help us obtain the countable compactness in small powers.

\begin{notation} Given $m \in \omega$ and a free Abelian group $G$,  define ${\mathcal F}(G,m)$, as the set of $m$-uples $(f_{0} , \ldots , f_{m-1})$, where  $A$ is an infinite subset of $\omega$, with $f_i:\,A\longrightarrow G$ for each $i<m$ and for each nonzero function $s:m \longrightarrow {\mathbb Z}$ there exists a finite set $F_s$ such that the sequence $(\sum_{i<m}s(i).f_i(n):\, n\in A \setminus F_s)$ is one-to-one.

 Define ${\mathcal F}(G,<\kappa)=\bigcup_{m<\kappa} {\mathcal F}(G,m)$ for $\kappa \leq \omega$.

If $G={\mathbb Z}^{( {\mathfrak c} \times \omega)}$ then we will write  ${\mathcal F}(<\kappa)$ and $ {\mathcal F}(m)$.
 \end{notation}

 \begin{notation}
Given  $F\in {\mathbb Z}^{({\mathfrak c}\times \omega)}$ and a family $\{z_{\alpha ,i}:\, (\alpha, i) \in \supp F\}$ of elements of some group $H$, $z_F$  denotes the sum
$\sum_{(\alpha, i) \in \supp F} F(\alpha,i).z_{\alpha,i}$. 

If $A$ is a subset of $\omega$, $g:\, A \longrightarrow {\mathbb Z}^{({\mathfrak c}\times \omega)}$ and $\{z_{\alpha ,i}:\, (\alpha, i) \in \bigcup \{\supp g(n):\, n \in A\} \}$, we denote by $z_g$ the function with domain $ A$ and range $ H$ such that $z_g(n)=z_{g(n)}$, for each $n \in A$.
\end{notation}

Given $M\in \omega$, the set ${\mathcal F}( <M+1)$ will be used often to index ${\mathcal F}(K, <M+1)$, when $K$ is a free Abelian group:

\begin{lemma} \label{lem.basis.to.basis}
Let $X=\{ x_{\alpha, i}:\, \alpha < {\mathfrak c} \text{ and } i<\omega\}$ be a set of generators for an Abelian group $G$.

Then  ${\mathcal F}(G, <M+1)\subseteq\{ (x_{g_{i}})_{i<m}:\, (g_i)_{i<m} \in {\mathcal F}( <M+1) \}$. 
\end{lemma}

\begin{proof} Let $(y_i)_{i<m}$ be an element of $ {\mathcal F}(G, <M+1)$ with each $y_i$ with domain $A$, for each $i<M$. Since $X$ generates $G$, there exists $g_i :\, A\longrightarrow {\mathbb Z}^{(\mathfrak c\times \omega)}$ such that
	$y_i=x_{g_i}$, for each $i<m$.
	
	Fix a nonzero function $s:m \longrightarrow {\mathbb Z}$. By hypothesis, there exists a finite set $F_s$ such that the sequence $(\sum_{i<m}s(i).y_i(n):\, n\in A \setminus F_s)$ is one-to-one.
	
	Let $\{\chi_{\alpha,i}:\, \alpha < {\mathfrak c} \text{ and } i< \omega\}$ be the canonical basis of ${\mathbb Z}^{({\mathfrak c} \times \omega)}$ and $\Phi$ be the homomorphism from  ${\mathbb Z}^{({\mathfrak c} \times \omega)}$ onto $G$ such that $\Phi(\chi_{\beta,j})=x_{\beta,j}$, for each $\beta <{\mathfrak c}$ and $j<\omega$.
	Now, $\sum_{i<m}s(i).y_i(n)=\sum_{i<m}s(i).x_{g_i}(n)=$
	 
	  $\sum_{i<m}s(i).\sum_{(\beta, j) \in \supp g_i(n)}{g_i}(n)(\beta ,j).x_{\beta, j}=$
	  
	  $\sum_{i<m}s(i).\sum_{(\beta, j) \in \supp g_i(n)}{g_i}(n)(\beta ,j).\Phi(\chi_{\beta, j})=$
	  
	  $\Phi(\sum_{i<m}s(i).\sum_{(\beta, j) \in \supp g_i(n)}{g_i}(n)(\beta ,j).\chi_{\beta, j})=$
	  
	   $\Phi(\sum_{i<m}s(i).{g_i}(n)).$
	   
	   Since $\Phi$ is a function, it follows that 
	   
	   $(\sum_{i<m}s(i).{g_i}(n):\, n \in A \setminus F_s)$ is one to one.

	   Therefore, $(g_i)_{i<m}\in {\mathcal F}(<M+1)$.	  
\end{proof}

The next lemma shows that the sequences in ${\mathcal F}(G,<M+1)$ are sufficient to deal with the countable compactness of the $M$-th power of a free Abelian group $G$, for each positive integer $M$.

\begin{lemma} \label{lemma.reduction.free.abelian.group} Let $M$ be a positive integer, $G$ be a topological free Abelian group and $(f_i)_{i<M}$ be a sequence of functions of domain $A$ in ${\mathcal F}(G,<M+1)$ such that the sequence $(({f_i(n)})_{i<k}:\, n\in A)$ has an accumulation point in $G^k$. Then $G^M$ is countably
compact.
\end{lemma}

\begin{proof} Fix arbitrarily $g_i:\, \omega \longrightarrow G$ for each $i<M$.
	If there exists $B\in [\omega]^\omega$ such that each $g_i|_B$ is constant for each $i<k$, we are done. So, we can assume w.l.o.g. that there exists $B\in \omega$ infinite such that $g_j|_B$ is one-to-one, for some $j<M$.
	
Fix a free ultrafilter ${\mathcal U}$  on $\omega$ such that $B\in {\mathcal U}$. The Abelian group $K$ generated by $\{ [g_i]_{\mathcal U}:\, i<M\}$
is torsion-free and finitely generated. Thus, it is isomorphic to a finite sum of copies of ${\mathbb Z}$. Let $K_0$ be the subgroup of $K$ by the classes of constant sequences. Since it is a subgroup of a free Abelian group, $K_0$ is a
free Abelian group. We claim that $K/K_0$ is also torsion-free. Indeed, if it was not torsion-free then there exits a positive integer $l $ and $h \in K \setminus K_0$  such that $l.h \in K_0$.  Let $g\in \langle \{g_i:\, i<M\}\rangle $ be such that $h=[g]_{\mathcal U}$. Then, there exists $B \in {\mathcal U}$ such that $( l.g(n):\, n \in B)$ is a constant sequence. Consequently, as a subset of a free Abelian group, the sequence $(g(n):\, n\in B)$ is constant. Therefore $h=[g]_{\mathcal U}\in K_0$, a contradiction. From the claim above and the fact that $K/K_0$ is finitely generated, it follows that $K/K_0$ is a free Abelian group. Let $\{t_i:\, i < k'\}$ with $k'\leq M$ be such that $\{[t_i]_{\mathcal U} +K_0:\, i<k'\}$ is a basis for $K/K_0$. Let $\{ t_i:\, k' \leq i<k'' \}$,  with $k'\leq k''\leq M$ be constant functions such that $\{[t_i]_{\mathcal U}:\, k'\leq i<k''\}$ is a basis for $K_0$. It is straightforward to see that $\{ [t_i]_{\mathcal U}:\, i<k''\}$ is a basis for $K$.

Let $C_0$ be an element of ${\mathcal U}$ such that for each $i<k$ there exists $s_i:\,k'' \longrightarrow {\mathbb Z}$ such that $g_i(n) =\sum_{j<k''}s_i(j).t_j(n)$ for each $n\in C_0$.

Enumerate all functions $r:\, k' \longrightarrow {\mathbb Z}$ that are not constantly $0$ as $\{ r_m:\,  m \in \omega\}$.
For each $m\in \omega$ and for each $C\in {\mathcal U}$ we have that $\{\sum_{j<k'} r_m(j).t_j(n):\, n \in C\}$ is infinite.

Choose inductively $n_q > \text{ max } \{n_0, \ldots , n_{q-1}\}$ such that $n_q$ is an element of

$\bigcap_{a,b\leq q , p <q} ( C_0 \setminus\{ n \in \omega :\, \sum_{j <k'} r_a(j).t_j (n)=\sum_{j <k'} r_b(j).t_j (n_p)\})$.

 This is possible, since

 $\{ n\in \omega :\, \sum_{j <k'} r_a(j).t_j (n)=\sum_{j <k'} r_b(j).t_j (n_p)\}$

 \noindent is not an element of ${\mathcal U}$ for each  $a,b \leq q$ and $p<q$.

The set  $A= \{n_q:\, q\in \omega \}$ is so that $ \{ \sum_{j<k'} r_q(j).t_j(n) :\, n \in A \setminus n_q \}$
is one-to-one for each $q \in \omega $. Indeed, if $u> v\geq q$ then $n_u \in C_0 \setminus \{n \in  \omega:\, \sum_{j<k'}r_q(j).t_j(n) =\sum_{j<k'}r_q(j).t_j(n_v)\}$. Hence,
$\sum_{j<k'}r_q(j).t_j(n_u) \neq\sum_{j<k'}r_q(j).t_j(n_v)$.

Therefore $(t_j|_A)_{ j<k'}$  is in ${\mathcal F}(G, k')$. By hypothesis, $( (t_j(n))_{j<k'}:\, n \in A )$ has an accumulation point  in $G^{k'}$. As $ t_j|_A$ is constant for $k' \leq j <k''$, it follows that there exists an ultrafilter ${\mathcal V}$ on $A$ and
$(d_j)_{j<k''}$ in $G^{k''}$ that is the ${\mathcal V}$-limit of $( (t_j(n))_{j<k''}:\, n \in A )$.

Since $A$ is a subset of $C_0$, the previouly defined $s_i:\,k'' \longrightarrow {\mathbb Z}$, for each $i<k$, are such that $g_i(n) =\sum_{j<k''}s_i(j).t_j(n)$ for each $n\in A$. Therefore $\{(g_{i}(n))_{i<k}:\, n \in A\}$ has $(\sum_{j<k''}s_i(j).d_j)_{i<k}$ as an accumulation point in $G^k$.
\end{proof}

\subsection{Accumulation points}

We will prove some auxiliary results to show that sequences associated to ${\mathcal F}(<k)$ have many accumulation points that are related to independent sets. This property is used to preserve accumulation points after refining the topology.

\begin{lemma} \label{nontorsion.tree.first.step} Let $H$ be a topological Abelian group without non-trivial convergent sequences such that
$H^M$ is countably compact for some positive integer $M$, $A _0$ and $A_1$ are infinite subsets of $\omega$ and  $\{y_{i,n}:\, i<M \mbox{ and } n \in A_0 \cup A_1\}$ is a subset of $H$. Suppose that, for each $j<2$, the function $s_j:\, M \longrightarrow  {\mathbb Z}$  and the finite set $F_j$ are such that the sequence $\{ \sum_{i <M} s_j(i).y_{i,n}:\, n \in A_j\setminus F_j\}$ is one-to-one.

If $\{(y_{i,n})_{i<M}:\, n \in A_j\}$ has an accumulation point in an open set $U_j$ of $H^M$ then, for each $j<2$, there exist an infinite subset $B_j$ of $A_j$ and a basic open set $V_j=\prod_{i<M}V_{j,i}$  subset of $U_j$  such that

 \begin{enumerate}
\item $\overline{V_j}\subseteq U_j$, for each $j<2$;
  \item $ \overline{ \{(y_{i,n})_{i<M}:\, n \in B_j\} }\subseteq V_j$, for each $j<2$;

  \item $\overline{\sum_{i\in \supp s_0}s_0(i).V_{0,i}} \cap \overline{ \sum_{i\in \supp s_1}s_1(i). V_{1,i}} = \emptyset$ and

\item $0\notin \overline{\sum_{i\in \supp s_0}s_0(i).V_{0,i}} \cup \overline{ \sum_{i\in \supp s_1}s_1(i). V_{1,i}}$.
 \end{enumerate}
\end{lemma}

\begin{proof} 
	
	Fix $j<2$. By hypothesis, the sequence
	$((y_{i,n})_{i<M}:\, n \in A_j)$
	has an accumulation point $(a_i)_{i<M} \in U_j$. Let $O_j$ be a open neighborhood of $(a_i)_{i<M}$ such that $\overline{O_j}\subseteq U_j$.    
	
	The set $C_j=\{ n \in A_j \setminus F_j:\, (y_{i,n})_{i<M} \in O_j\}$ is infinite and the sequence $( \sum_{i\in \supp s_j }s_j(i).y_{i,n}:\, n \in C_j)$ is one-to-one. These sequences have  ${\mathfrak c}$ many accumulation points, for each $j<2$.  We can fix $c_0, c_1\in H$ such that $c_0\neq c_1\neq 0 \neq c_0$ and
	$c_j$ is an accumulation point of $\{ \sum_{i\in \supp s_j}s_j(i).y_{i,n}:\, n \in C_j\}$, for each $j<2$. Let $W_j$ be a neighborhood of $c_j$, for each $j<2$,
	such that  $\overline{W_0}\cap \overline{W_1}=\emptyset$ and $0\notin \overline{W_0}\cup \overline{W_1}$. Let $W_j^*$ be an open set such that $c_j \in W_j^* \subseteq \overline{W_j^*} \subseteq W_j$, for each $j<2$.
	
	 The set $D_j=\{ n \in C_j:\, \sum_{i\in \supp s_j}s_j(i).y_{i,n} \in W_j^*\}$ is infinite and $((y_{i,n})_{i<M}:\, n \in D_j) \subseteq O_j$ has an accumulation point $(b_{j,i})_{i<M} \in \overline{O_j}\subseteq U_j$. 
	
	It follows from the definition of $D_j$ that
	$\sum_{i\in \supp s_j}s_j(i).b_{j,i} \in \overline{W_j^*} \subseteq W_j$.
	Using the continuity of the addition, it follows that there exists a basic open set $V_j=\prod_{i<M}V_{j,i}$ such that
	$(b_{j,i})_{i<M} \in  V_j \subseteq \overline{V_j} \subseteq U_j$ and 
	$\sum_{i\in \supp s_j}s_j(i).V_{j,i}  \subseteq W_j$.  Therefore, $\overline{\sum_{i\in \supp s_j}s_j(i).V_{j,i}} \subseteq  \overline{ W_j}$. By the properties of $\overline{W_j}$ for $j<2$, it follows that conditions $3)$ and $4)$ are satisfied. The sets $V_j$ for $j<2$ were defined to satisfy condition $1)$.
	
	Let $V_j^*$ be an open set such that $(b_{j,i})_{i<M} \in  V_j^* \subseteq \overline{V_j^*}\subseteq V_j$, for each $j<2$.
	The set $B_j=\{n \in D_j:\, (y_{i,n})_{i<M} \in V_j^*\}$, for $j<2$, satisfies condition $2)$.
 \end{proof}

\begin{proposition} \label{nontorsion.tree.general.step}

 Let $H$ be a topological Abelian group without non-trivial convergent sequences, $M$ be a positive integer such that
$H^M$ is countably compact and  $((y_{i,n})_{i<M}:\,  n \in A) \in {\mathcal F}(H,M)$. 

There exists a family $\{a_{i,f}:\, i<M \text{ and } f \in \ ^\omega 2\}$ in $H$ such that

$A)$ $(a_{i,f}:\, i<M)$ is an independent set whose elements have infinite order, for each $f \in \ ^\omega 2$;

$B)$ $(a_{i,f})_{i<M}$ is an accumulation point of $( (y_{i,n})_{ i<M }:\, n \in A)$, for each $f \in \ ^\omega 2$ and

$C)$ $\sum_{i\in \supp s_0}s_0(i).a_{i,f_0}\neq \sum_{i\in \supp s_1}s_1(i).a_{i,f_1}$,  whenever $s_j:\, M \longrightarrow {\mathbb Z}$ of nonempty support for each $j<2$ and  $f_0,f_1\in \ ^\omega 2$  with $f_0\neq f_1$.
\end{proposition}
\begin{proof}
We will construct a tree of basic open subsets of $H^M$, $\{ V_{p}:\, p \in \bigcup _{k<\omega} \ ^k 2\}$ and a tree of subsets of $A$, $\{A_p:\, p \in \bigcup _{k<\omega} \ ^k 2\}$ satisfying the following:

$a)$ $\overline{V_{p^\wedge j}}\subseteq V_{p}$,  for each  $ p \in\bigcup _{k<\omega} \ ^k 2$ and $j<2$;

 $b)$  $ \overline{ \{(y_{i,n})_{i<M}:\, n \in A_p\} }\subseteq V_p$,  for each $ p \in \bigcup _{k<\omega} \ ^k 2$;

\smallskip

In $c)$ and $d)$ below, we recall that $V_{p,i}$ is the $i$-th coordinate of the basic open set $V_{p}$, for each $i<M$:

\smallskip

 $c)$ $\overline{\sum_{i\in \supp s_0}s_0(i).V_{p_0,i}} \cap \overline{ \sum_{i\in \supp s_1}s_1(i). V_{p_1,i}} = \emptyset$, whenever
$|p_0|=|p_1|=k$, $s_j:\, M \longrightarrow [-k, k] \cap {\mathbb Z}$ has nonempty support and
$p_0\neq p_1$;

\smallskip

In $c)$ and $d)$ below, we recall that $V_{p,i}$ is the $i$-th coordinate of the basic open set $V_{p}$, for each $i<M$:

\smallskip

$d)$ $0\notin \overline{\sum_{i\in \supp s}s(i).V_{p,i}}$, whenever $ p \in\bigcup _{k<\omega} \ ^k 2$ and $s:\, M \longrightarrow [-|p|, |p|] \cap {\mathbb Z}$  has nonempty support and

$e)$ $A_{p^ \wedge j}\subseteq A_p$, for each $p \in \bigcup _{k<\omega} \ ^k 2$ and $j<2$.

\medskip

Let us assume for a moment that there are sets satisfying $a)$-$e)$. By condition $e)$,  we can find $A_f \subseteq A$ such that $A_f \setminus A_{f|_n}$ is finite, for each $f \in \ ^\omega 2$ and for each $n \in \omega$.

Let $(a_{i,f})_{i<M}$ be an accumulation point of $((y_{i,n})_{ i<M }:\, n \in A_f)$. It follows that condition $B)$ is satisfied.

Fix a non-constantly $0$ function $s:\, M \longrightarrow {\mathbb Z}$.  Choose $k \in {\mathbb N}$ such that $\text{ ran } s \subseteq [-k,k]$. Then $\sum_{i\in \supp s}s(i).a_{i,f} \in \overline{\sum_{i\in \supp s}s(i).V_{f|_{k},i}}$ by condition $b)$. By condition $d)$, this last set
does not contain $0$, therefore, $\sum_{i\in \supp s}s(i).a_{i,f}\neq 0$. As $s$ is arbitrary, it follows that $(a_{i,f}:\, {i<M})$ is linearly independent and condition $A)$ is satisfied.

 Fix functions of nonempty support $s_j:\, M \longrightarrow {\mathbb Z}$ and $f_j \in \ ^\omega 2$ for $j<2$ with $f_0 \neq f_1$. Let $k \in {\mathbb N}$ be such that $\text{ ran } s_0 \cup \text{ ran } s_1 \subseteq [-k,k]$ and $f_0 |_k \neq f_1|_k$.

By condition $c)$, it follows that

$\overline{\sum_{i\in \supp s_0}s_0(i).V_{f_0 |_k,i}} \cap \overline{ \sum_{i\in \supp s_1}s_1(i). V_{f_1|_k,i}} = \emptyset$.

It follows by condition $b)$ that, for each $j<2$,

\noindent
 $\sum_{i\in \supp s_j} s_j(i).a_{i,f}\in $ $\overline{\sum_{i\in \supp s_j}s_j(i).V_{f_j |_k,i}}$. 
 
 Therefore $\sum_{i\in \supp s_0}s_0(i).a_{i,f_0}\neq \sum_{i\in \supp s_1}s_1(i).a_{i,f_1}$  and condition $C)$ is satisfied. Hence all conditions $A)$-$C)$ are satisfied.

\smallskip

 We will now go back to the construction of the sets satisfying condition $a)-e)$. Set $V_\emptyset =G$ and $A_\emptyset = A$. Clearly all conditions are satisfied.

Suppose that $V_p$ and $A_p$ are constructed for each $p \in \bigcup_{k<m}  \ ^k 2$ and satisfy the inductive conditions $a)$-$e)$.

For each $p \in \ ^{m-1} 2$, set $A_{p^ \wedge 0}^0=A_{p^\wedge 1}^0=A_p$ and $V_{p^ \wedge 0 }^0=V_{p^\wedge 1}^0=V_p$. Enumerate as $\{ ((s^0_j , p^0_j), (s^1_j , p^1_j)):\, j<t\}$ all the pairs
$((s_0, p_0) , (s_1,p_1))$ such that $p_i \in \ ^m 2$ with $p_0\neq p_1$ and $s_ u :\, M \longrightarrow [-m,m] \cap {\mathbb Z}$ whose support is nonempty, for each $u<2$.

Suppose we have defined $A_p^j$ and $V_p^j$ for each $j<l \leq t$ satisfying:

$I)$ $\overline{V_p ^{j+1}}\subseteq V_p ^j$ and $A_p^{j+1} \subseteq A_p^j$ if $p \in \{p_j^0, p_j^1\}$ and $j+1<l$;

$II)$ $V_p ^{j+1}= V_p ^j$ and $A_{p}^{j+1}=A_{p}^j$ if $p \in \ ^m 2 \setminus \{p_j^0, p_j^1\}$ and $j+1<l$;

$III)$ $ \overline{ \{(y_{i,n})_{i<M}:\, n \in A_p^j\} }\subseteq V_p^j$, for each $p \in \ ^m 2$ and  $j<l$;

 $IV)$ $\overline{\sum_{i\in \supp s_j^0}s_j^0(i).V^{j+1}_{p^0_j,i}} \cap \overline{\sum_{i\in \supp s_j^1}s_j^1(i).V^{j+1}_{p^1_j,i}} = \emptyset$,  for each $j+1< l$ and

$V)$ $0\notin \overline{\sum_{i\in \supp s_j^0}s_j^0(i).V^{j+1}_{p_j^0,i}} \cup \overline{ \sum_{i\in \supp s_j^1}s_j^1(i). V^{j+1}_{p_{j}^1 ,i}}$.

By  $III)$, we can apply Lemma \ref{nontorsion.tree.first.step} on $A_r=A_{p_{l-1}^r}^{l-1}$ and $U_r=V_{p_{l-1}^r}^{l-1}$
to obtain $A_{p_{l-1}^r}^l \subseteq A_{p_{l-1}^r}^{l-1}$ and a basic open set $ V_{p_{l-1}^r}^l$ for each $r<2$ that satisfy:

$i)$ $\overline{V_{p_{l-1}^r}^l}\subseteq V_{p_{l-1}^r}^{l-1}$ for each $r<2$;

$ii)$ $ \overline{ \{(y_{i,n})_{i<M}:\, n \in A_{p_{l-1}^r}^l\} }\subseteq V_{p_{l-1}^r}^l$, for each $r<2$;

$iii)$ $\overline{\sum_{i\in \supp s^0_{l-1}}s^0_{l-1}(i).V_{p_{l-1}^0,i}^l} \cap \overline{ \sum_{i\in \supp s^1_{l-1}}s^1_{l-1}(i). V_{p_{l-1}^1,i}^l} = \emptyset$ and

$iv)$ $0\notin \overline{\sum_{i\supp s^0_{l-1}}s^0_{l-1}(i).V_{p_{l-1}^0,i}^l} \cup \overline{ \sum_{i\in \supp s^1_{l-1}}s^1_{l-1}(i). V_{p_{l-1}^1,i}^l}$.

Condition $I)$ holds by $i)$. Define $V_p ^{l}= V_p ^{l-1}$ and $A_{p}^{l}=A_{p}^{l-1}$ if $p \in \ ^m 2 \setminus \{p_{l-1}^0, p_{l-1}^1\}$ as in condition $II)$. Condition $III)$ for $p \in \{p_{l-1}^0, p_{l-1}^1\}$
is satisfied by $ii)$. Condition $III)$ for $p \in \ ^m 2 \setminus \{p_{l-1}^0, p_{l-1}^1\}$ is satisfied by condition $II)$ and condition $III)$ for $j=l-1$. Conditions $IV)$ and $V)$ for $j=l-1$ follows from conditions $iii)$ and $iv)$.

Now, define $V_p=V_p^t$ and $A_p=A_p^t$. We will check that conditions $a)$-$e)$ are satisfied.
 For each $p \in \ ^m 2$, there exists $j< t$ such that $p=p^0 _j$. By $I)$ and $II)$ it follows that

$\overline{V_{p}}\subseteq \overline{V_p^{j+1}}\subseteq V_p^j \subseteq V_{p|_{m-1}}$ and $a)$ holds.

It follows from condition $III)$ that
 $ \overline{ \{(y_{i,n})_{i<M}:\, n \in A_p^t\} }\subseteq V_p^t$, for each $ p \in \bigcup _{k<\omega} \ ^k 2$. Thus
condition $b)$ is satisfied.

 Fix a pair $(s_0,p_0), (s_1,p_1)$ with $p_0\neq p_1$ and $|p_0|=|p_1|=m$, $s_r:\, M \longrightarrow [-m, m] \cap {\mathbb Z}$  whose support is nonempty, for each $r<2$. Let $j<t$ be such that $(s_r,p_r)=(s^r_j,p^r_j)$, for each $r<2$. By condition $IV)$

 $\overline{\sum_{i\in \supp s^0_j}s^0_j(i).V_{p^0_{j+1},i}} \cap \overline{ \sum_{i\in \supp s^1_j}s^1_j(i). V_{p^1_{j+1},i}} = \emptyset$. By condition $I)$ and $II)$ it follows that $V_{p_r}=V_{p^r_{t}} \subseteq V_{p_{j+1}}^r  $, for each $r<2$. Thus $c)$ is satisfied.

It follows from condition $V)$ that

\noindent $0\notin \overline{\sum_{i\in \supp s_j^0 }s_j^0(i).V^{j+1}_{p_j^0,i}} \cup \overline{ \sum_{i\in \supp s_j^1 }s_j^1(i). V^{j+1}_{p_{j}^1 ,i}}$. By $I)$ and $II)$ it follows that condition $d)$ is satisfied.

Condition $I)$, $II)$ and $A_p^0=A_{p|_{m-1}}$ imply that $A_{p}\subseteq A_{p|_{m-1}}$, for each $p \in \bigcup _{m<\omega} \ ^m 2$ and $j<2$, therefore condition $e)$ is satisfied.
\end{proof}

\begin{proposition}\label{proposition.free.abelian.subgroup} Let $M$ be a positive integer and $\{ (g_{\alpha ,i})_{i<M_\alpha}:\, \omega \leq \alpha <{\mathfrak c} \}$ be an enumeration of ${\mathcal F}( <M+1)$ such that
	$\bigcup_{i<M_\alpha, n\in A_\alpha} \supp g_{\alpha, i}(n) \subseteq \alpha \times \omega$, for each $\alpha< {\mathfrak c}$.
	
	 Suppose that there exists an infinite non torsion topological group $H$ without non-trivial convergent sequences and $H^M$ is countably compact. Then, there exists a independent set $X=\{x_{\alpha ,i}:\, i<\omega \text{ and } \alpha <{\mathfrak c}\}\subseteq H$ such that $(x_{\alpha,i})_{ i<M_\alpha}$ is an accumulation point of the sequence $( (x_{g_{\alpha,i}(n)})_{i<M_\alpha}:\, n \in A_\alpha)$, for each infinite ordinal $\alpha < {\mathfrak c}$.
\end{proposition}

\begin{proof}

	A countably compact Abelian non torsion group is such that its free rank is at least ${\mathfrak c}$. In particular, we can find an independent set $\{ x_{m,i}:\, m, i <\omega \}$ of infinite order. 
	
	Suppose by induction that
$\{ x_{\beta , i }:\, \beta < \alpha  \text{ and } i<\omega \}$ is linearly independent with $\alpha <{\mathfrak c}$.
Let $K$ be the group generated by $X_\alpha=\{ x_{\beta , i }:\, \beta < \alpha  \text{ and } i<\omega \}$.

Then, the sequence $\{(x_{g_{\alpha ,i}}(n))_{i<M_\alpha}:\,  n<\omega \}$ is already defined. Set $y_{\alpha,i}=x_{g_{\alpha ,i}}$, for each $i<M_\alpha$. By the linear independence of $X_\alpha$, it follows from an argument similar to Lemma \ref{lem.basis.to.basis} that  $(y_{\alpha,i})_{i<M_\alpha} \in {\mathcal F}(K,M+1)$. Apply  Lemma \ref{nontorsion.tree.general.step} to obtain
 $\{a_{f,i}:\, f \in \ ^\omega 2 \text{ and } i<M_\alpha \}$.

We claim that there exists $h \in \ ^\omega 2$ such that $X_\alpha \cup \{a_{h,i}:\, i<M_\alpha\}$ is independent. Suppose by contradiction that this was not the case. Then, there exists $r_f:\, M_\alpha \longrightarrow {\mathbb Z}$ of nonempty support such that $ \sum_{i<M_\alpha} r_f(i).a_{f,i} $ is an element
of $K$, for each $f \in \ ^\omega 2$.
By cardinality arguments, it follows that there exists two distinct $f_0$ and $f_1$ in $\ ^\omega 2$ such that
 $\sum_{i<M_\alpha} r_{f_0}(i).a_{f_0,i} = \sum_{i<M_\alpha} r_{f_1}(i).a_{f_1,i} $, but this contradicts
condition $C)$ in Lemma  \ref{nontorsion.tree.general.step}.

Set $x_{\alpha, i}=a_{h,i}$, for each $i<M_\alpha$. Choose $x_{\alpha,i}$ for $M_\alpha \leq i <\omega$
so that $\{ x_{\beta , i }:\, \beta \leq \alpha  \text{ and } i<\omega \}$ is a basis for a free Abelian group.
\end{proof}

For the case $\omega$, it is not necessary to refine the topology.

\begin{theorem}  Suppose that there exists a non torsion topological Abelian group $H$ without non-trivial convergent sequences such that the group $H^m$ is countably compact, for each $m<\omega$. Then, there exists a topological free Abelian group $G$ of cardinality ${\mathfrak c}$ such that $G^m$ is countably compact for each $m<\omega$ and $G^\omega$ is not countably compact.
\end{theorem}

\begin{proof} Let $\{ g_{\alpha ,i}:\, \omega \leq \alpha <{\mathfrak c} \text{ and } i < M_\alpha \}$ be an enumeration of ${\mathcal F}(<\omega)$ such that
$\bigcup_{i<M_\alpha} \supp g_{\alpha, i} \subseteq \alpha \times \omega$, for each $\alpha \in [\omega , {\mathfrak c}[$.

Using the same argument in the proof of Proposition \ref{proposition.free.abelian.subgroup}, it follows that
 there exists an independent set $X=\{x_{\alpha ,i}:\, i<\omega \text{ and } \alpha <{\mathfrak c}\}$ such that $(x_{\alpha,i})_{ i<M_\alpha}$ is an accumulation point of the sequence $( (x_{g_{\alpha,i}(n)})_{i<M_\alpha}:\, n \in \omega )$, for each infinite ordinal $\alpha < {\mathfrak c}$.

Let $G$ be the group generated by $X$. By Lemma \ref{lem.basis.to.basis}, every element of ${\mathcal F}(G,<\omega)$ appears listed as $ (x_{g_{\alpha,i}})_{i<M_\alpha}$, for some $\alpha \in [\omega, {\mathfrak c}[$.

 By Lemma \ref{lemma.reduction.free.abelian.group}, it follows that every finite power of the group $G$ is countably compact. Tomita \cite{To98} showed that the countable power of a topological free Abelian group cannot be countably compact and this ends the proof.
\end{proof}

\subsection{Refining topologies}

The next result is used to refine the original topology. A lemma for groups of order $2$ appears in Tomita \cite{To05fm}. There are some differences as we are dealing with a group of infinite order. For the sake of completeness we give a detailed proof.

\begin{proposition}\label{proposition.refining.free.abelian.subgroup} Let $\{ (g_{\alpha ,i})_{i<M_\alpha}:\, \omega \leq \alpha <{\mathfrak c} \}$ be an enumeration of ${\mathcal F}(<M+1)$ such that
$\bigcup_{i<M_\alpha} \supp g_{\alpha, i} \subseteq \alpha \times \omega$ for each $\alpha< {\mathfrak c}$ and $X=\{x_{\alpha ,i}:\, i<\omega \text{ and } \alpha <{\mathfrak c}\}$ be a basis for a free Abelian group as in Proposition \ref{proposition.free.abelian.subgroup}. There exists $\{ z_{\alpha,i}:\, \alpha <{\mathfrak c} \text{ and } i\in \omega \} \subseteq \langle X \rangle \times {\mathbb T}^{{\mathfrak c}}$ such that

$i)$ $z_{\alpha,i}$ extends $ x_{\alpha, i}$, for each $\alpha <{\mathfrak c} \text{ and } i<\omega$;

$ii)$ $(z_{\alpha,i})_{ i<M_\alpha}$ is an accumulation point of the sequence $( (z_{g_{\alpha,i}(n)})_{i<M_\alpha}:\, n \in \omega )$, for each infinite ordinal $\alpha < {\mathfrak c}$;

$iii)$ For each function $D:\, \omega\times \omega \longrightarrow {\mathbb T}$, there exists a coordinate $\mu$ such that $ z_{n,i}(\mu)=D(n,i)$ for each $(n,i) \in \omega \times \omega$;

$iv)$ If ${\mathcal U}$  is a free ultrafilter,  $F_j$ is a finite nonempty subsets of $\omega$ for $j<2$, $r_j:\, F_j \longrightarrow {\mathbb Z}\setminus \{0\}$ for $j<2$ are distinct, then  ${\mathcal U}$-lim $(\sum_{n\in F_0}r_0(n).z_{n,i}:\, i \in \omega ) \neq {\mathcal U}$-lim $(\sum_{n\in F_1}r_1(n).z_{n,i}:\, i \in \omega )$ and

$v)$ If ${\mathcal U}_j$ for $j<2$ are distinct ultrafilters, $F_j$ for $j<2$ are nonempty finite subsets of $\omega$ and $r_j:\, F_j \longrightarrow {\mathbb Z}\setminus \{0\}$ for each $j<2$, then ${\mathcal U}_0$-lim $(\sum_{n\in F_0)}r_0(n).z_{n,i}:\, i \in \omega )\neq {\mathcal U}_1$-lim $(\sum_{n\in F_1}r_1(n).z_{n,i}:\, i \in \omega )$.
\end{proposition}

\begin{proof} Let $A_\alpha$ be the domain of $f_{0,\alpha}$ and ${\mathcal V}_\alpha$ be a free ultrafilter on $A_\alpha$ such that
$(x_{\alpha,i})_{ i<M_\alpha}$ is the ${\mathcal V}_\alpha$-limit of the sequence $((x_{g_{\alpha,i}(n)})_{i<M_\alpha}:\, n \in A_\alpha)$, for each infinite ordinal $\alpha < {\mathfrak c}$.

Enumerate ${\mathbb T}^{\omega \times \omega}$ as $\{D_\mu:\, \mu <{\mathfrak c}\}$.

We will define a homomorphism $\phi_\mu:\, {\mathbb Z}^{({\mathfrak c} \times \omega)} \longrightarrow {\mathbb T}$ that will witness $iii)$, for  $D_\mu:\, \omega \times \omega \longrightarrow {\mathbb T}$.

First, set $\phi_\mu(x_{n,i})=D_\mu(n,i)$, for each $(n,i) \in \omega \times \omega$  and $\phi_\mu(x_{\alpha,i})=0$
for each $\alpha \in [\omega , {\mathfrak c}[$ and $i\geq M_\alpha$.

 Now, use the divisibility of ${\mathbb T}$, the independence of $X$  and the fact that $\supp g_{\alpha ,i}(n) \subseteq \alpha \times \omega$, for each $n\in A_\alpha$, to inductively extend it to a homomorphism such that $\phi_\mu(x_{\alpha,i})$ is the ${\mathcal V}_\alpha$-limit of the sequence $(\phi_\mu (x_{g_{\alpha,i}(k)}):\, k \in A_\alpha )$, for each infinite ordinal $\alpha < {\mathfrak c}$ and for each $i<M_\alpha$.

 Let $z_{\alpha,i}$ be the function $x_{\alpha,i} \ ^{\wedge} \{ \phi_\mu (x_{ \alpha ,i}):\, \mu <{\mathfrak c}\}$, for each $\alpha < {\mathfrak c}$ and $i<\omega$. Clearly condition $i)$ and $iii)$ are satisfied.

Condition $ii)$ is also satisfied, as $(z_{\alpha,i})_{ i<M_\alpha}$ is the ${\mathcal V}_\alpha$-limit of the sequence $( (z_{g_{\alpha,i}(k)})_{i<M_\alpha}:\, k\in A_\alpha)$, for each infinite ordinal $\alpha < {\mathfrak c}$. 

To check the two last conditions, we first fix an independent set $\{ a_n:\, n \in \omega \}\subseteq {\mathbb T}$  whose elements have infinite order.

We will check condition $iv)$ first. Fix an ultrafilter ${\mathcal U}$, finite nonempty subsets $F_j$ of $\omega$ for $j<2$ and  distinct functions $r_j:\, F_j \longrightarrow {\mathbb Z}\setminus \{0\}$ for $j<2$.

  Fix a function $D$  such that $D(n,i)=a_n$ for each $ n \in  F_0 \cup F_1$ and $i\in \omega$ and let $\mu<{\mathfrak c}$ be such that $D=D_\mu$. Thus, the ${\mathcal U}$-limit of $(\sum_{n\in F_j}r_j(n).z_{n,i}(\mu):\, i \in \omega )$ is $\sum_{n \in F_j}r_j(n).a_n$.
By the linearly independence of $\{a_n:\, n \in F_0\cup F_1\}$, it follows that 
$\sum_{n \in F_0}r_0(n).a_n\neq \sum_{n \in F_1}r_1(n).a_n$. Hence, condition $iv)$ is satisfied.

Finally, we check condition $v)$. Let $\{ {\mathcal U}_j:\, j<2\}$ be distinct ultrafilters, $F_j$ finite nonempty subsets of $\omega$ and $r_j:\, F_j \longrightarrow {\mathbb Z}\setminus \{0\}$ for each $j<2$. Fix $m_j \in \omega$ such that $m_j \in F_j$ and $B_0,B_1$ are disjoint subsets of $\omega$ such that $B_j \in {\mathcal U}_j$ for $j<2$. Set the function $E:\, \omega \times \omega \longrightarrow {\mathbb T}$ such that $E(n,i)=a_j$ if $(n,i)\in \{m_j\}\times B_j$ for $j<2$ and  $E(n,i)=0$ otherwise. Let $\nu<{\mathfrak c}$ be such that $E=D_\mu$.

It follows that 
 ${\mathcal U_0}$-lim $(\sum_{n\in F_0}r_0(n).z_{n,i}(\nu):\, i \in \omega )=r_0(m_0).a_0\neq r_1(m_1).a_1={\mathcal U_1}$-lim $(\sum_{n\in F_1}r_1(n).z_{n,i}(\nu):\, i \in \omega )$. Hence condition $v)$ is also satisfied.
\end{proof}

We will now take a subgroup of the refinement to obtain the desired topology for finite powers.

\begin{theorem} \label{theorem.freeabelian.mcc.m+1not} Suppose that there exists an infinite non torsion topological group $H$ without non-trivial convergent sequences and such that $H^M$ is countably compact for some positive integer $M$. Then there exists a free Abelian group $G$ such that $G^M$ is countably compact and $G^{M+1}$ is not countably compact.
\end{theorem}

\begin{proof} Let $\{ g_{\alpha ,i}:\, \omega \leq \alpha <{\mathfrak c} \text{ and } i < M_\alpha \}$ be an enumeration of ${\mathcal F}(<M+1)$ such that
$\bigcup_{i<M_\alpha} \supp g_{\alpha, i} \subseteq \alpha \times \omega$ for each $\alpha \in [\omega , {\mathfrak c}[$ and each element of ${\mathcal F}(<M+1)$ appears ${\mathfrak c}$ many times in the enumeration.

By Propositions \ref{proposition.free.abelian.subgroup} and \ref{proposition.refining.free.abelian.subgroup} there exist a family $Z=\{ z_{\alpha,i}:\, \alpha <{\mathfrak c} \text{ and } i\in \omega \}$ such that

$a)$ $\{ z_{\alpha,i}:\, \alpha <{\mathfrak c} \text{ and } i<\omega\}$ is a basis for a free Abelian;

$b)$ $(z_{\alpha,i})_{ i<M_\alpha}$ is an accumulation point of the sequence $((z_{g_{\alpha,i}(n)})_{i<M_\alpha}:\, n \in \omega )$, for each infinite ordinal $\alpha < {\mathfrak c}$;

$c)$ For each function $D:\, \omega\times \omega \longrightarrow {\mathbb T}$, there exists a coordinate $\mu$ such that $ z_{n,i}(\mu)=D(n,i)$ for each $(n,i) \in \omega \times \omega$ and

$d)$ If ${\mathcal U}_0$ and ${\mathcal U}_1$ are distinct free ultrafilters. $F_0$ and $F_1$ are nonempty finite subsets of $\omega$ and $r_j:\, F_j \longrightarrow {\mathbb Z}\setminus \{0\}$ for each $j<2$, then  ${\mathcal U}_0$-lim $(\sum_{n\in F_0}r_0(n).z_{n,i}:\, i \in \omega )\neq {\mathcal U}_1$-lim $(\sum_{n\in F_1}r_1(n).z_{n,i}:\, i \in \omega)$.

\medskip

Let ${\mathfrak U}$ be the set of all ultrafilters ${\mathcal U}$ such
that the ${\mathcal U}$-limit of $(z_{n,i}:\, i \in \omega )$ is an element of the group $\langle Z \rangle$, for every $n<M+1$. Note that if $(a_n)_{n<M+1}$ is an accumulation point of the sequence $((z_{n,i})_{i<M+1}:\, i \in \omega )$ then there is ${\mathcal U}\in {\mathfrak U}$ such that $(a_n)_{n<M+1}$ is the ${\mathcal U}$-limit of this sequence. 

We will now obtain inductively a subset $I_\alpha \subseteq {\mathfrak c}\times \omega$ for each $\alpha \in [\omega , {\mathfrak c}[$ and the desired example will be $\langle \{z_{\beta ,i}:\, (\beta,i) \in \bigcup_{\alpha <{\mathfrak c} } I_\alpha \}\rangle$.

The induction will satisfy the following conditions:

$1)$ $|I_\alpha|\leq |\alpha|+\omega$ for each $ \alpha < {\mathfrak c}$ and $\{I_\alpha:\, \alpha <{\mathfrak c}\}$ is a $\subseteq$-increasing chain in ${\mathcal P}({\mathfrak c}\times \omega)$;

$2)$ ${\mathfrak U}_\alpha $ is the set of all ultrafilters ${\mathcal U} \in {\mathfrak U} $ for which
there exists $r:\, F \longrightarrow {\mathbb Z}\setminus \{0\}$ with $\emptyset \neq F\subseteq M+1$ such that the ${\mathcal U}$-limit of $\{ \sum_{m\in F} r(m).z_{m,i}:\, i\in \omega \}$ is an element of $\langle \{ z_{\beta , i}:\,  (\beta,i) \in I_\alpha \}\rangle$;

$3)$ $|{\mathfrak U}_\alpha|\leq |\alpha|+\omega$ for each $ \alpha <{\mathfrak c}$ and $\{ {\mathfrak U}_\alpha:\, \alpha <{\mathfrak c}\}$ is $\subseteq$-increasing chain;

$4)$ $J_\alpha$ is a subset of ${\mathfrak c}\times \omega$ and for each
${\mathcal U}\in {\mathfrak U}_{\alpha}$, there exists $m <M+1$ such that ${\mathcal U}$-lim $\{z_{m,i}:\, i \in \omega \} \notin \langle \{z_{\beta ,i }:\, (\beta,i) \in {\mathfrak c} \setminus J_\alpha \}\rangle$;

$5)$ $|J_\alpha| \leq |\alpha|$ for each $\omega \leq \alpha < {\mathfrak c}$ and $\{ J_\alpha :\, \alpha <{\mathfrak c}\}$ is a $\subseteq$-increasing chain in ${\mathcal P}({\mathfrak c}\times \omega)$;

$6)$ $\theta_\alpha \in [\omega, {\mathfrak c}[$ is the least ordinal $\theta \in {\mathfrak c}\setminus \{ \theta_\beta:\, \beta < \alpha \}$ such that $\bigcup_{n\in A_\theta , i <M_\theta} \text{ supp } g_{\theta , i}(n) \subseteq \bigcup_{\mu <\alpha }I_\mu$ for each $\alpha \in [\omega , {\mathfrak c}[$;

$7)$ $\rho_\alpha \in [\omega, {\mathfrak c}[$ is such that $(g_{\theta_\alpha,i})_{i<M_{\theta_\alpha}}=(g_{\rho_\alpha,i})_{i<M_{\rho_\alpha}}$, for each $\omega \leq \alpha <{\mathfrak c}$;

$8)$ $\{\rho_\alpha\} \times \omega  \subseteq {\mathfrak c} \setminus ( \bigcup_{\mu <\alpha} (J_\mu \cup I_\mu))$, for each $\omega \leq \alpha < {\mathfrak c}$;

$9)$ $\{\rho_\alpha\} \times M_{\rho_\alpha} \subseteq I_\alpha$, for each $\omega \leq \alpha <{\mathfrak c}$ and

$10)$ $J_\alpha \cap I_\alpha=\emptyset$, for each $\alpha <{\mathfrak c}$.

\medskip

We will show first that if  $I_\alpha$, $J_\alpha$, $\theta_\alpha$, $\rho_\alpha$, ${\mathfrak U}_\alpha$ are constructed for $\omega \leq \alpha <{\mathfrak c} $ satisfying the conditions $1)-10)$, then $G=\langle \{ z_{\beta, i}:\,  (\beta ,i ) \in \bigcup_{\alpha <{\mathfrak c}}I_\alpha\}\rangle$ is as required.

\smallskip

By condition $a)$, it follows that $G$ is a free Abelian group.

We claim that $G^M$ is countably compact. By Lemma \ref{lemma.reduction.free.abelian.group}, it suffices to show that every sequence associated to an element of  ${\mathcal F}(G,m)$ has an accumulation point in $G^m$, for each $m \leq M$. Fix an arbitrary $m\leq M$ and $(f_i)_{i<m} \in {\mathcal F}(G,m)$.
 Let $A$ be the domain of $f_i$ for some (all) $i<m$. By Lemma \ref{lem.basis.to.basis}, there exists $(h_i)_{i<m} \in {\mathcal F}(m)$  such that $f_i(n)=z_{h_i(n)}$, for each $i<m$ and  $n\in A$. Let $\mu<{\mathfrak c}$ be the smallest ordinal such that $\bigcup_{i<m , n\in A }h_i(n) \subseteq \bigcup_{\beta < \mu} I_\beta$ and let $\theta$ be an ordinal such that $A_\theta=A$, $M_\theta=m$ and
$(h_0, \ldots , h_{m-1})=(g_{\theta,0}, \ldots, g_{\theta,m-1})$.

We claim that $\theta = \theta_\nu$ for some infinite ordinal $\nu$. If that is not the case, $\theta_\beta < \theta$ for each $\beta \geq\mu$, but this contradicts $6)$ since the $\theta_\beta$'s are
pairwise distinct.
Thus, there exists $\nu<{\mathfrak c}$ such that $\theta=\theta_\nu$. By condition 7), $M_{\theta_\nu}=M_{\rho_\nu}$ and $g_{\theta_\nu,i}=g_{\rho_\nu,i}$
for each $i<m$ and $(\rho_\nu,i) \in I_\nu$. By $b)$, the sequence $ ((f_i(n))_{i<m}:\, n\in A) =   ((z_{g_{\rho_\nu,i}}(n)_{i<m}:\, n\in A)$ has $(z_{\rho_\nu,i})_{i<m} \in G^m$ as an accumulation point.

We will show now that $((z_{n,i})_{n<M+1}:\, i \in \omega)$ does not have an accumulation point in $G^{M+1}$.
Suppose by way of contradiction that it does have an accumulation point. Then, there exists an ultrafilter ${\mathcal U}$ such that each sequence $(z_{n,i}:\, i \in \omega)$ has a ${\mathcal U}$-limit in $G$, for each $n<M+1$.
 In particular, it follows that ${\mathcal U}\in {\mathfrak U}$. Let $\alpha$ be an ordinal such that the ${\mathcal U}$-limit of $(z_{0,i}:\, i \in \omega)$ is an element of the group generated by $\{ z_{\beta ,i}:\, (\beta ,i) \in I_\alpha \}$. It follows by condition $2)$ that ${\mathcal U}\in {\mathfrak U}_\alpha$. By $10)$, there exists $k<M+1$ such that the ${\mathcal U}$-limit of
$(z_{k,i}:\, i \in \omega )$ is not an element of $\langle \{z_{\beta ,i }:\, \beta \in {\mathfrak c} \setminus J_\alpha \}\rangle$.

By $1)$, $5)$ and $10)$, it follows that $G\subseteq \langle \{z_{\beta ,i }:\, (\beta ,i) \in {\mathfrak c} \setminus J_\gamma \}\rangle$. Therefore, the ${\mathcal U}$-limit of $(z_{k,i}:\, i \in \omega ) $ is not in $ G$, a contradiction.

\smallskip

We will now start the inductive construction satisfying conditions $1)-10)$.

\bigskip

{\bf Case 1} $[\alpha < \omega]$. Set $I_\alpha=\omega\times \omega$, for each $\alpha\in \omega$. Then condition $1)$ holds.

We claim that ${\mathfrak U}_\alpha$ defined as in condition $2)$ is the empty set. Let ${\mathcal U}\in {\mathfrak U}$ and $z \in \langle \{z_{n,i}:\, i \in \omega \text{ and } n \in \omega\}\rangle$. 

 {\em Case A}. The element $z$ is not the identity, let  $F\subseteq \omega\times \omega$ be a nonempty finite set and $r:\, F \longrightarrow {\mathbb Z}\setminus \{0\}$ be such that $z= \sum_{(n,i)\in F}r(n,i).z_{n,i}$.  Let $\{a_{n,i}:\, (n,i)\in F\}$ be an independent set whose elements have infinite order in ${\mathbb T}$. Let $D:\, \omega \times \omega \longrightarrow {\mathbb Z}$ be such that $D(n,i)=a_{n,i}$ for each $(n,i) \in F$ and $D(n,i)=0$ for each $(n,i) \in \omega \times \omega \setminus F$. Let $\mu<{\mathfrak c}$ be as in condition $c)$ for $D$.

Then the ${\mathcal U}$-limit of $( z_{n,i}(\mu):\, i \in \omega )$ is $0$ for each $n<M+1$ and $z(\mu)= \sum_{(n,i)\in F}r(n,i).a_{n,i}\neq 0$. Hence, the  ${\mathcal U}$-limit of the sequence $( \sum_{n<M+1} s(n). z_{n,i}(\mu):\, i \in \omega)$ is $0$ for every  $s:\, M+1 \longrightarrow {\mathbb Z}$. 

Therefore,
$z$ is not the ${\mathcal U}$-limit of any sequence $( \sum_{n\in \text{ dom } s} s(n). z_{n,i}:\, i \in \omega)$, for every $s:\, M+1 \longrightarrow {\mathbb Z}$ whose support is nonempty.

\smallskip

{\em Case B}.  The element $z$ is the identity. Fix 
$s:\, M+1 \longrightarrow {\mathbb Z}$ whose support is nonempty  and $a\in {\mathbb T}$ any non torsion element. Let $n_*<M+1$ be such that $s(n_*)\neq 0$. Let $E:\, \omega\times \omega \longrightarrow {\mathbb T}$ be such that $E(n_*,i)=a$ for each $i\in \omega$ and $E(n,i)=0$ for each $n \in \omega\setminus \{n_*\}$ and $i \in \omega$. Let $\nu<{\mathfrak c}$ be as in condition $c)$ for $E$. Then, the ${\mathcal U}$-lim$( \sum_{n \in M+1} s(n).z_{n,i}(\nu):\, i \in \omega )=$
${\mathcal U}$-limit of $( s(n_*).z_{n_*,i}(\nu):\, i \in \omega )=s(n_*).a \neq 0=z(\mu)$.

\medskip

 It follows from Cases A and B that ${\mathcal U} \notin {\mathfrak U}_\alpha$, for each $\alpha\in \omega$. Hence condition $2)$ and $3)$ are satisfied.

 Set $J_\alpha=\emptyset$, for each $\alpha\in \omega$. Condition $4)$ is satisfied as ${\mathfrak U}_\alpha=\emptyset$. As $J_\alpha=\emptyset$, it follows that condition $5)$ and $10)$ are satisfied. Conditions $6)$- $9)$ are trivially satisfied since $\alpha<\omega$.

\smallskip

Suppose that the induction holds for every $\mu < \alpha$.

\bigskip

{\bf Case 2}  $[\alpha \geq \omega]$.

Let $\theta_\alpha$ be the least ordinal for which $\bigcup_{ i< M_{\theta_\alpha} , \,  n \in \omega} \text{ supp } g_{\theta_\alpha,i}(n) \subseteq \bigcup_{\mu <\alpha}I_\mu$. Then condition $6)$ is satisfied.

The set $  \{ \beta <{\mathfrak c}:\, (\exists i \in \omega) ((\beta , i) \in \bigcup_{\mu<\alpha}I_\mu) \}$ has cardinality smaller than ${\mathfrak c}$ and the set $\{ \rho \in [\omega , {\mathfrak c}[:\, M_\rho = M_{\theta_\alpha} \text{ and }  g_{\rho,i}=g_{\theta_\alpha,i}\ \forall i<M_{\theta_\alpha}\}$ has cardinality ${\mathfrak c}$. Therefore, we can choose  $\rho_\alpha$ satisfying condition $7)$ and $8)$. Set $I_\alpha=\bigcup_{\mu <\alpha}I_\mu \cup (\{\rho_\alpha \} \times M_{\rho_\alpha})$. Then, condition $1)$ and $9)$ are satisfied.

Let ${\mathfrak U}_\alpha$ be as in condition $2)$. We claim that $|{\mathfrak U}_\alpha|\leq |\alpha|$. In fact,
for each ${\mathcal U}\in {\mathfrak U}_\alpha$, there exists a function of nonempty support $s:\, M+1 \longrightarrow {\mathbb Z}$ such that ${\mathcal U}$-limit of $( \sum_{k\in \supp s} s(k).z_{k,i}:\, i \in\omega )$ is $ y_{\mathcal U} \in\langle \{ z_{\beta , i}:\, (\beta,i) \in I_\alpha\}\rangle$.

 By condition $d)$, it follows that $\{ y_{\mathcal U}:\, {\mathcal U} \in {\mathfrak U}_\alpha\}$ are pairwise distinct. Therefore, $|{\mathfrak U}_\alpha | \leq |I_\alpha| +\omega \leq |\alpha| +\omega$ by condition $1)$. This shows that condition $3)$ holds.

We will define $J_\alpha $:

Fix ${\mathcal U} \in 
 {\mathfrak U}_\alpha\setminus \bigcup_{ \mu <\alpha} {\mathfrak U}_\mu$. Let $z_{{\mathcal U},n}$ be the $ {\mathcal U}$-limit of $( z_{n,i}:\, i \in \omega ) $, for each $ n <M+1$.
We claim that there is $n_{\mathcal U} <M+1$ such that
$z_{{\mathcal U},n_{\mathcal U}} \notin \langle \{ z_{\beta,i}:\, (\beta,i) \in I_\alpha \}\rangle $.
Suppose by contradiction that this is not the case. Then, there exists $w_n \in \langle \{ z_{\rho_\alpha, i}:\, i <M_{\rho_\alpha} \}\rangle$ such that $z_{{\mathcal U},n}-w_n \in \langle \{ z_{\beta ,i}:\, (\beta ,i ) \in \bigcup_{\mu<\alpha} I_\mu\}\rangle$, for each $n<M+1$. Since $\{w_n:\, n <M+1\}$ is a subset of a free Abelian group generated by at most $M$ elements,
it follows that there exists a function of nonempty support   $r:\, M+1 \longrightarrow {\mathbb Z}$ such that $\sum_{n\in \supp r} r(n).w_n=0$. Thus, $ {\mathcal U}$-limit of $(\sum_{n\in \supp r}r(n) z_{n,i}:\, i \in \omega ) =\sum_{n\in \supp r} r(n).z_{{\mathcal U},n} = \sum_{n\in \supp r} r(n).(z_{{\mathcal U},n} - w_n)\in \langle \{ z_{\gamma,i}:\, (\gamma,i) \in \bigcup_{\mu<\alpha} I_\mu\} \}\rangle $.
But this contradicts the fact that ${\mathcal U} \notin \bigcup_{\mu<\alpha}{\mathfrak U}_\mu$.

Let $F_{\mathcal U}$ be a finite subset of ${\mathfrak c}\times \omega$ and $r_{\mathcal U}:\, F_{\mathcal U} \longrightarrow {\mathbb Z}\setminus \{0\}$ be such that
$z_{{\mathcal U},n_{\mathcal U}}= z_{r_{\mathcal U}}$. Fix $(\beta_{\mathcal U}, i_{\mathcal U}) \in F_{\mathcal U} \setminus  I_\alpha$.

Define $J_\alpha$ as the set $\bigcup_{\mu<\alpha}J_\mu \cup  \{ (\beta_{\mathcal U},i_{\mathcal U}):\, {\mathcal U} \in {\mathfrak U}_\alpha \setminus \bigcup_{ \mu <\alpha}{\mathfrak U}_\mu \}$. By condition $3)$, it follows that $|J_\alpha|\leq |\alpha|$ and condition $5)$ holds. 

Clearly condition $4)$ holds for each ${\mathcal U} \in {\mathfrak U}_\alpha \setminus \bigcup_{ \mu <\alpha}{\mathfrak U}_\mu$.
Since $J_\alpha \supseteq J_\mu$ for $\mu<\alpha$, it follows from $4)$ at stage $\mu$ that condition $4)$ also holds for ${\mathfrak U}_\mu$ at stage $\alpha$. Thus, condition $4)$ holds for each element of ${\mathfrak U}_\alpha$.

Finally, to show that condition $10)$ holds, note that $J_\alpha \cap I_\alpha$ is the union of three empty sets $[(\bigcup_{\mu <\alpha} J_\mu)\cap (\bigcup_{\nu<\alpha}I_\mu)] \cup [(\bigcup_{\mu<\alpha} J_\mu)\cap (I_\alpha \setminus (\bigcup_{\nu<\alpha}I_\nu))] \cup [J_\alpha \setminus (\bigcup_{\mu<\alpha} J_\mu) \cap I_\alpha]$. The first intersection is empty by inductive hypothesis, the second intersection is empty by the choice of $\rho_\alpha$ and the third intersection is empty by the choice of $(\beta_{\mathcal U}, i_{\mathcal U})$ for each ${\mathcal U} \in {\mathfrak U}_\alpha \setminus \bigcup_{ \mu <\alpha}{\mathfrak U}_\mu$.
\end{proof}

\begin{example} Assume the existence of ${\mathfrak c}$ incomparable selective ultrafilters. For each $M <\omega$, there exists a topological group topology on the free Abelian group $G$ of cardinality ${\mathfrak c}$ such that $G^M$ is countably compact and
$G^{M+1}$ is not countably compact.
\end{example}

\begin{proof} Assuming the existence of ${\mathfrak c}$ incomparable selective ultrafilters, Tomita \cite{To?} showed that there exists a topological free Abelian group without non-trivial convergent sequences whose every finite power is countably compact. Applying Theorem \ref{theorem.freeabelian.mcc.m+1not}, we obtain the example.
\end{proof}

\section{Countably compact groups of finite order}

The proof for the finite case is very similar to the non torsion case but we have the infinite case that did not occur for the non torsion case. A key difference is that the ultrapower of a  free Abelian group is no longer a free Abelian group, whereas an ultrapower  of a vector space over the field ${\mathbb Z}_p$ is a vector space over the field ${\mathbb Z}_p$. This last fact is used to show that an arbitrary countable family of sequences can be associated to a family of sufficiently independent sequences for which we can refine the topology as in the torsion case.

We will sketch the proofs for the torsion case and point out some differences with the non torsion case.

 \begin{notation}
  Given a prime number $P$, an Abelian group $G$ of order $P$ and $\lambda\leq\omega$, let ${\mathcal F}_P(G,\lambda)$ be the set of all sequences $(f_i)_{i < \lambda}$ from some infinite subset $A\subseteq \omega$ into $G$  such that, for  a nonempty support function  $s \in ({\mathbb Z}_P)^{(\lambda)}$, there exists $F_s$ finite such that the set $\{ \sum_{i \in \text{ supp } s}s(i).f_i(n):\, n\in A \setminus F_s\}$ is one-to-one.

 Define ${\mathcal F}_P(G, <\kappa)=\bigcup_{\lambda<\kappa} {\mathcal F}_P(G,\lambda)$, for each $\kappa \leq \omega_1$.
 
 If $G={\mathbb Z}_P^{( {\mathfrak c} \times \omega)}$ then we will write  ${\mathcal F}_P(<\kappa)$ and $ {\mathcal F}_P(\kappa)$,  respectively.
 
\end{notation}

\begin{notation}
	Given a prime number $P$, $F\in {\mathbb Z}_P^{({\mathfrak c}\times \omega)}$, a group $H$ of order $P$ and a family $\{z_{\alpha ,i}:\, (\alpha, i) \in \supp F\} \subseteq H$, $z_F$  denotes the sum
	$\sum_{(\alpha, i) \in \supp F} F(\alpha,i).z_{\alpha,i}$. 
	
	Given $A$ a subset of $\omega$, $g:\, A \longrightarrow {\mathbb Z}_P^{({\mathfrak c}\times \omega)}$ and $\{z_{\alpha ,i}:\, (\alpha, i) \in \bigcup \{\supp g(n):\, n \in A\} \} \subseteq H$, we denote by $z_g$ the function with domain $ A$ and range $ H$ such that $z_g(n)=z_{g(n)}$, for each $n \in A$.
\end{notation}

\begin{lemma} \label{lemma.p.subgroup} Let $\kappa \leq \omega_1$ be a cardinal and $H$ be a topological group whose torsion subgroup is of uncountable cardinality and such that $H^\gamma$ is countably compact, for each cardinal $\gamma< \kappa$. Then
 there exists a prime number $P$ and a group topology without non-trivial convergent sequences on $({\mathbb Z}_P)^{({\mathfrak c})}$ such that its $\gamma$-th power is countably compact, for each $\gamma <\kappa$.
\end{lemma}

\begin{proof} By hypothesis,  the torsion subgroup of $H$ is uncountable. Hence, there exists $m$ such that $m.g=0$ for uncountably many $g\in H$. Let $K_0=\{g \in H:\, m.g=0\}$. The group $K_0$ is a bounded torsion group, therefore, a direct sum of cyclic groups of some order $p^l$, $p$ prime, and $p^l$ divisor of $m$. Then, there exists a prime $P$ that divides $m$ and uncountably many  summands that are copies of ${\mathbb Z}_{P^l}$ for some $l$ such that $P^l$ divides $m$. It  follows that $K_1=\{ g \in H:\, P.g=0\}$ is uncontable. The group $K_1$ is a closed subgroup of $H$, thus the $\gamma$-th power of $K_1$ is countably compact, for each cardinal $\gamma <\kappa$.
	 Using closing off arguments and the fact that  $\gamma < \kappa \leq\omega_1$, we can find a subgroup $K$ of $K_1$  of cardinality ${\mathfrak c}$ such that $K^\gamma$ is countably compact for each $\gamma<\kappa$.
	
	The group $H$ does not have non-trivial convergent sequences, therefore, its subgroup  $K$ also inherits this property. Since $K$ is a group of order prime $P$, it is isomorphic to $({\mathbb Z}_P)^{({\mathfrak c})}$. Use the isomorphism to give the desired group topology on  $({\mathbb Z}_P)^{({\mathfrak c})}$.
\end{proof}

\begin{lemma} \label{lem.basis.to.basis.P}
	Let $X=\{ x_{\alpha, i}:\, \alpha < {\mathfrak c} \text{ and } i<\omega\}$ be a set of generators for an Abelian group $G$ of prime order $P$.
	
	Then  ${\mathcal F}_P(G, <\kappa)\subseteq\{ (x_{g_{i}})_{i<\lambda}:\, (g_i)_{i<\lambda} \in {\mathcal F}_P( <\kappa) \}$. 
\end{lemma}

\begin{proof} The proof is as in Lemma \ref{lem.basis.to.basis}.

	\end{proof}

\begin{lemma} \label{lemma.reduction.torsion.abelian.group} Let $\kappa\leq \omega_1$ and  $G$ be a topological Abelian group of prime order $P$ such that $\{ (f_i(n))_{i<\lambda}:\, n\in A\}$ has an accumulation point in $G^\lambda$ for each
 $(f_i)_{i<\lambda} \in {\mathcal F}_P(G,<\kappa) $ with $\text{ dom } f_i=A$ for each $i<\lambda$.  Then $G^\gamma$ is countably compact, for each $\gamma <\kappa$.
\end{lemma}

\begin{proof}  Fix $\alpha <\kappa$ and  arbitrary functions $g_i:\, \omega \longrightarrow G$,  for each $i<\alpha$.
	We want to show that $((g_i(n))_{i<\alpha}:\, n \in \omega)$ has an accumulation point in $G^\alpha$.
Fix a free ultrafilter ${\mathcal U}$  on $\omega$. Then, the Abelian group $K$ generated by $\{ [g_i]_{\mathcal U}:\, i<\alpha\}$ is a vector space of order $P$ of finite (possibly $0$) or countable infinite dimension. If the dimension is $0$ then the ${\mathcal U}$-limit of $(g_i(n):\, n \in \omega)$ is $0$, for each $i<\alpha$ and we are done. Assume that the dimension is positive and let $K_0$ be the subgroup of $K$ of the classes of constant sequences. Since $K$ is a vector space over the field ${\mathbb Z}_P$, there exists a subgroup $K_1$ of $K$ such that $K$ is the direct sum of $K_0$ and $K_1$. Let $\{ t^j_i :\, i<\lambda_j\}$ be such that $\{ [t^j_i]_{\mathcal U} :\, i<\lambda_j\}$ is a basis for $K_j$, for each $j<2$. We can choose $t^0_i$ to be constant functions, for each $i<\lambda_0$.

Let $C_m$ be an element of ${\mathcal U}$  for each $m<\alpha$, be such that there exists $s^j_m:\,\lambda_j \longrightarrow {\mathbb Z}$ with finite support for each $j<2$ such that $g_m(n) =\sum_{l\in\supp s^0_m}s^0_m(l).t^0_l(n)+\sum_{l\in \supp s^1_m}s^1_m(l).t^1_l(n)$, for each $n\in C_m$. Without loss of generality, we can assume that $\{ C_m :\, m \in \omega \}$ is a $\subseteq$-decreasing sequence.

 Enumerate as $\{r_m:\, m \in \omega \}$ the set of all functions $r:\, \alpha_1 \longrightarrow {\mathbb Z}_P$ with nonempty finite support. Note that the set $\{\sum_{j\in \supp r} r(j).t^1_j(n):\, n \in C\}$ is infinite, for each $C\in {\mathcal U}$.

Choose inductively $n_q > \text{ max } \{n_0, \ldots , n_{q-1}\}$ such that
 $n_q \in$ 
 
 \noindent
  $ \bigcap_{a,b \leq q \wedge p<q} C_q\setminus \{ n \in \omega :\, \sum_{l \in \supp r} r_a(l).t^1_l (n)=\sum_{l \in \supp r} r_b(l).t^1_l (n_p) \}$.

The set  $A= \{n_q:\, q\in \omega \}$ is such that $ ( \sum_{l\in \supp r_q} r_q(l).t^1_l(n) :\, n \in A \setminus n_q )$
is one-to-one for each $q \in \omega $. Therefore $( t^1_i)_{ i<\lambda_1}$ is in ${\mathcal F}_P(G,\lambda_1)$, where the domain of each $t^1_i$ is $A$. By hypothesis, $( (t_i(n))_{i<\lambda_1}:\, n \in A )$ has an accumulation point  in $G^{\lambda_1}$.

The proof of some arguments above, as well as rest of the proof is as in the proof of Lemma \ref{lemma.reduction.free.abelian.group}.
\end{proof}

\begin{lemma} \label{torsion.tree.first.step} Let $H$ be a topological group  of prime order $P$ without non-trivial convergent sequences such that
$H^\lambda$ is countably compact for some $\lambda\leq \omega$, $A_0$ and $A_1$ infinite subsets of $\omega$ and  $\{y_{i,n}:\, i<\lambda \mbox{ and } n \in A_0 \cup A_1\}$ be a subset of $H$. Suppose that the function $s_j:\, \lambda \longrightarrow  {\mathbb Z}_P$ has finite support and the finite set $F_j$ are such that $\{ \sum_{i \in \supp s_j} s_j(i).y_{i,n}:\, n \in A_j\setminus F_j\}$ is one-to-one,   for each $j<2$.

If $\{(y_{i,n})_{i<\lambda}:\, n \in A_j\}$ has an accumulation point in an open set $U_j$ of $H^\lambda$ then there exist $B_j$  an infinite subset $A_j$ and $V_j=\prod_{i<\lambda}V_{i}^j$ a basic open subset of $U_j$ for each $j<2$ such that

 \begin{enumerate}

\item $\overline{V_j}\subseteq U_j$ for each $j<2$;
  \item $ \overline{ \{(y_{i,n})_{i<\lambda}:\, n \in B_j\} }\subseteq V_j$, for each $j<2$;

  \item $\overline{\sum_{i\in\supp s_0}s_0(i).V_{0,i}} \cap \overline{ \sum_{i\in\supp s_1}s_1(i). V_{1,i}} = \emptyset$ and

\item $0\notin \overline{\sum_{i\in \supp  s_0}s_0(i).V_{0,i}} \cup \overline{ \sum_{i\in \supp s_1}s_1(i). V_{1,i}}$.
 \end{enumerate}
\end{lemma}

\begin{proof}
The proof is as in Lemma \ref{nontorsion.tree.first.step} if we replace $M$ by $\lambda$. 
\end{proof}

\begin{proposition} \label{torsion.tree.general.step}

 Let $H$ be a topological group of order $P$ without non-trivial convergent sequences such that
$H^\gamma$ is countably compact, for some cardinal $0<\gamma \leq\omega$. Let $((y_{i,n})_{i<\lambda}:\,  n \in A) \in {\mathcal F}_P(H,\lambda)$, for $\lambda \leq \gamma$. There exists a family $\{a_{i,f}:\, i<\lambda \text{ and } f \in \ ^\omega 2\} \subseteq H$ such that

$A)$ $\{a_{i,f}:\, i<\lambda\}$ is an independent set whose elements have order $P$, for each $f \in \ ^\omega 2$;

$B)$ $(a_{i,f})_{i<\lambda}$ is an accumulation point of $\{(y_{i,n})_{ i<\lambda} :\,  n \in A\}$, for each $f \in \ ^\omega 2$ and

$C)$ $\sum_{i\in \supp s_0}s_0(i).a_{i,f_0}\neq \sum_{i\in \supp s_1}s_1(i).a_{i,f_1}$, for  $s_j:\, \lambda\longrightarrow {\mathbb Z}_P$ with nonempty finite support for $j<2$ and $f_0,f_1 \in \ ^\omega 2$ with $f_0 \neq f_1$.
\end{proposition}

\begin{proof} We modify the proof of Proposition \ref{nontorsion.tree.general.step} and give here a sketch and indicate where changes had to be made.
	
	We will construct a tree of basic open subsets of $H^\lambda$, $\{ V_{p}:\, p \in \bigcup _{k<\omega} \ ^k 2\}$ and a tree of subsets of $A$, $\{A_p:\, p \in \bigcup _{k<\omega} \ ^k 2\}$ satisfying conditions $a)-d)$. Conditions $c)$ and $d)$ are different from the proof of Proposition \ref{nontorsion.tree.general.step}.
	
	$a)$ $\overline{V_{p^\wedge j}}\subseteq V_{p}$,  for each  $ p \in\bigcup _{k<\omega} \ ^k 2$ and $j<2$;
	
	$b)$  $ \overline{ \{(y_{i,n})_{i<\lambda}:\, n \in A_p\} }\subseteq V_p$,  for each $ p \in \bigcup _{k<\omega} \ ^k 2$;
	
		\smallskip
	
	In $c)$ and $d)$ below, we recall that $V_{p,i}$ is the $i$-th coordinate of the basic open set $V_{p}$, for each $i<\lambda$:
	
	\smallskip
	
	$c)$ $\overline{\sum_{i\in \supp s_0}s_0(i).V_{p_0,i}} \cap \overline{ \sum_{i\in \supp s_1}s_1(i). V_{p_1,i}} = \emptyset$, whenever
	$|p_0|=|p_1|=k$, $s_j:\, \lambda \longrightarrow  {\mathbb Z}_P$ has non-empty support contained in $k$  and
	$p_0\neq p_1$;

	$d)$ $0\notin \overline{\sum_{i\in \supp s}s(i).V_{p,i}}$, whenever $ p \in\bigcup _{k<\omega} \ ^k 2$ and $s:\, \lambda \longrightarrow  {\mathbb Z}_P$  has nonempty support contained in $k \cap \lambda$ and
	
	$e)$ $A_{p^ \wedge j}\subseteq A_p$, for each $p \in \bigcup _{k<\omega} \ ^k 2$ and $j<2$.
	
	\medskip
	
	Let us assume for a moment that there are sets satisfying $a)$-$e)$. By condition $e)$,  we can find $A_f \subseteq A$ such that $A_f \setminus A_{f|_n}$ is finite, for each $f \in \ ^\omega 2$ and for each $n \in \omega$.
	
	Since $H^\gamma$ is countably compact and $\lambda\leq \gamma$, there exists  $(a_{i,f})_{i<\lambda}$ in $H^\lambda$ that is an accumulation point of $\{(y_{i,n})_{ i<\lambda }:\, n \in A_f\}$. It then follows that condition $B)$ is satisfied.
	
	Fix a function $s:\, \lambda \longrightarrow {\mathbb Z}$ whose support is a nonempty finite set.  Choose $k \in {\mathbb N}$ such that $\dom s \subseteq k$. Then $\sum_{i\in \supp s}s(i).a_{i,f} \in \overline{\sum_{i\in \supp s}s(i).V_{f|_{k},i}}$ by condition $b)$. The same argument in the proof of Proposition \ref{nontorsion.tree.general.step} shows that condition $A)$ is satisfied.
	
	Fix  $s_j:\, \lambda \longrightarrow {\mathbb Z}$ with nonempty finite support for $j<2$ and $f_0, f_1 \in \ ^\omega 2$  with $f_0 \neq f_1$. Let $k \in {\mathbb N}$ be such that $\dom s_0 \cup \dom s_1 \subseteq k$ and $f_0 |_k \neq f_1|_k$. The same argument now shows that condition $C)$ is satisfied.

	 Hence all conditions $A)$-$C)$ are satisfied.
	
	\smallskip
	
	We return to the construction of the sets satisfying condition $a)-e)$. Set $V_\emptyset =G$ and $A_\emptyset = A$. Clearly all conditions are satisfied.
	
	Suppose that $V_p$ and $A_p$ are constructed for each $p \in \bigcup_{k<m}  \ ^k 2$ and satisfy the inductive conditions $a)$-$e)$.
	
	For each $p \in \ ^{m-1} 2$, set $A_{p^ \wedge 0}^0=A_{p^\wedge 1}^0=A_p$ and $V_{p^ \wedge 0 }^0=V_{p^\wedge 1}^0=V_p$. Enumerate as $\{ ((s^0_j , p^0_j), (s^1_j , p^1_j)):\, j<t\}$ all the pairs
	$((s_0, p_0) , (s_1,p_1))$ such that $p_i \in \ ^m 2$ with $p_0\neq p_1$ and $s_ l :\, \lambda \cap m \longrightarrow  {\mathbb Z}_P$, for each $l<2$.
	
	Suppose we have defined $A_p^j$ and $V_p^j$ for each $j<l \leq t$ satisfying:
	
	$I)$ $\overline{V_p ^{j+1}}\subseteq V_p ^j$ and $A_p^{j+1} \subseteq A_p^j$ if $p \in \{p_j^0, p_j^1\}$ and $j+1<l$;
	
	$II)$ $V_p ^{j+1}= V_p ^j$ and $A_{p}^{j+1}=A_{p}^j$ if $p \in \ ^m 2 \setminus \{p_j^0, p_j^1\}$ and $j+1<l$;
	
	$III)$ $ \overline{ \{(y_{i,n})_{i<\lambda}:\, n \in A_p^j\} }\subseteq V_p^j$, for each $p \in \ ^m 2$ and  $j<l$;
	
	$IV)$ $\overline{\sum_{i\in \supp s_j^0}s_j^0(i).V^{j+1}_{p^0_j,i}} \cap \overline{\sum_{i\in \supp s_j^1}s_j^1(i).V^{j+1}_{p^1_j,i}} = \emptyset$,  for each $j+1< l$ and
	
	$V)$ $0\notin \overline{\sum_{i\in \supp s_j^0}s_j^0(i).V^{j+1}_{p_j^0,i}} \cup \overline{ \sum_{i\in \supp s_j^1}s_j^1(i). V^{j+1}_{p_{j}^1 ,i}}$.

	By  $III)$, we can apply Lemma
	\ref{torsion.tree.first.step}  on $A_r=A_{p_{l-1}^r}^{l-1}$ and $U_r=V_{p_{l-1}^r}^{l-1}$
	to obtain $A_{p_{l-1}^r}^l \subseteq A_{p_{l-1}^r}^{l-1}$ and a basic open set $ V_{p_{l-1}^r}^l$, for each $r<2$, that satisfy:
	
	$i)$ $\overline{V_{p_{l-1}^r}^l}\subseteq V_{p_{l-1}^r}^{l-1}$, for each $r<2$;
	
	$ii)$ $ \overline{ \{(y_{i,n})_{i<\lambda}:\, n \in A_{p_{l-1}^r}^l\} }\subseteq V_{p_{l-1}^r}^l$, for each $r<2$;
	
	$iii)$ $\overline{\sum_{i\in \supp s^0_{l-1}}s^0_{l-1}(i).V_{p_{l-1}^0,i}^l} \cap \overline{ \sum_{i\in \supp s^1_{l-1}}s^1_{l-1}(i). V_{p_{l-1}^1,i}^l} = \emptyset$ and
	
	$iv)$ $0\notin \overline{\sum_{i\supp s^0_{l-1}}s^0_{l-1}(i).V_{p_{l-1}^0,i}^l} \cup \overline{ \sum_{i\in \supp s^1_{l-1}}s^1_{l-1}(i). V_{p_{l-1}^1,i}^l}$.

	Conditions $I)-V)$ for $l$ follow as in the proof of Proposition \ref{nontorsion.tree.general.step} using conditions $i)-iv)$.

	Now, define $V_p=V_p^t$ and $A_p=A_p^t$. We will check that conditions $a)$-$e)$ are satisfied.
	For each $p \in \ ^m 2$, there exists $j< t$ such that $p=p^0 _j$. 
	
	Conditions $a)$ and $b)$ follow from conditions $I)-V)$ as in the proof of Proposition \ref{nontorsion.tree.general.step}.

	Fix a pair $(s_0,p_0), (s_1,p_1)$ with $p_0\neq p_1$ and $|p_0|=|p_1|=m$, $s_r:\, \lambda \longrightarrow  {\mathbb Z}_P$  whose support is a finite nonempty set, for each $r<2$. Let $j<t$ be such that $(s_r,p_r)=(s^r_j,p^r_j)$, for each $r<2$. Condition $c)$ holds applying the same argument as in Proposition \ref{nontorsion.tree.general.step}.

	Condition $d)$ and $e)$ also follow from the same arguments  in Proposition \ref{nontorsion.tree.general.step}.
\end{proof}

We state now the topology refinement for groups of prime order $P$.

\begin{proposition}\label{proposition.refining.P.subgroup} Let $L$ be a nonempty subset of $[\omega , {\mathfrak c}[$ such that ${\mathfrak c} \setminus L$ has cardinality ${\mathfrak c}$. Let $\kappa$ be a cardinal $\leq \omega_1$ and $\{ (g_{\alpha ,i})_{i< \lambda_\alpha}:\, \alpha \in [\omega , {\mathfrak c}[ \setminus L \}$ be an enumeration of ${\mathcal F}_P(<\kappa)$ such that
$\bigcup_{i<\lambda_\alpha} \supp g_{\alpha, i} \subseteq \alpha \times \omega$ for each $\alpha< {\mathfrak c}$. Let $X=\{x_{\alpha ,i}:\, i<\omega \text{ and } \alpha <{\mathfrak c}\}$ be a basis for a group of prime order $P$ such that  $(x_{\alpha,i})_{ i<\lambda_\alpha}$ is an accumulation point of the sequence $\{ (x_{g_{\alpha,i}(n)})_{i<\lambda_\alpha}:\, n \in \omega \}$, for each infinite ordinal $\alpha \in L$. There exists $\{ z_{\alpha,i}:\, \alpha <{\mathfrak c} \text{ and } i\in \omega \} \subseteq \langle X \rangle \times {\mathbb Z}_P ^{\mathfrak c}$ such that

$i)$ $z_{\alpha,i}$ extends $ x_{\alpha, i}$, for each $\alpha <{\mathfrak c} \text{ and } i<\omega$;

$ii)$ $(z_{\alpha,i})_{ i<\lambda_\alpha}$ is an accumulation point of the sequence $\{ (z_{g_{\alpha,i}(n)})_{i<\lambda_\alpha}:\, n \in \omega \}$, for each ordinal $\alpha \in [\omega, {\mathfrak c}[ \setminus L$;

$iii)$ For each function $D:\, L\times \omega \longrightarrow ({\mathbb Z}_P)^\omega$, there exists a coordinate $\mu$ such that $ z_{n,i}(\mu)=D(n,i)$, for each $(n,i) \in L \times \omega$;

$iv)$ If ${\mathcal U}$ is an ultrafilter, $F_j$ is nonempty finite subsets of $L$ for $j<2$ and $r_j:\, F_j \longrightarrow {\mathbb Z}_P\setminus \{0\}$ for $j<k$ are distinct, then the ${\mathcal U}$-lim $(\sum_{n\in F_j}r_j(n).z_{n,i}:\, i \in \omega ) \neq {\mathcal U}$-lim $(\sum_{n\in F_l}r_l(n).z_{n,i}:\, i \in \omega )$ and

$v)$ If ${\mathcal U}_0$ and ${\mathcal U}_1$ are distinct ultrafilters, $F_j$ is  a nonempty finite subset of $L$ for $j<2$ and $r_j:\, F_j \longrightarrow {\mathbb Z}_P\setminus \{0\}$ for each $j<2$, then ${\mathcal U}_0$-lim $(\sum_{n\in F_0}r_0(n).z_{n,i}:\, i \in \omega )\neq {\mathcal U}_1$-lim $(\sum_{n\in F_1}r_1(n).z_{n,i}:\, i \in \omega )$.
\end{proposition}

\begin{proof} It suffices to make minor changes in the proof of Proposition \ref{proposition.refining.free.abelian.subgroup}, replacing ${\mathbb Z}$ for ${\mathbb Z}_P$, ${\mathbb T}$ by $({\mathbb Z}_P)^\omega$  and $\omega \times \omega$ by $L\times \omega$. One can check that, following the proof of  Proposition \ref{proposition.refining.free.abelian.subgroup},  if $D(\beta,i)=0$ for $\beta \in L\cap \eta $ and $i \in \omega$
	and $D=D_\mu$ then $z_{\beta,i}(\mu)=0$, for every $\beta <\eta$ and $i\in \omega$.
\end{proof}

 We use $({\mathbb Z}_P)^\omega$ to get sufficiently large countable linearly independent subsets to mimick the proof using ${\mathbb T}$ for free Abelian groups.

\begin{theorem} \label{theorem.P.omegacc.omega1not.notch}  Suppose that there exists an infinite topological group $H$ without non-trivial convergent sequences of prime order $P$ and such that $H^\gamma$ is countably compact, for every $\gamma < \kappa$ with $\kappa  \leq  \omega_1$. Then there exists a topological group $G$ of order $P$ such that $G^\gamma$ is countably compact for each $\gamma <\kappa$ and $G^\kappa$ is not countably compact.
\end{theorem}

\begin{proof}

Let $\kappa \leq \omega_1$ and $L$ be a subset of cardinality $\kappa$. The only case we could not start with $L=\kappa$ is if $\kappa =\omega_1$ and the Continuum Hypothesis holds.

{\bf  Either $1)$ $\kappa <\omega_1$ or $2)$ $\kappa =\omega_1$ and $CH$ does not hold}.

 Take $L=\kappa$.

Let $Z=\{ z_{\alpha,i}:\, \alpha <{\mathfrak c} \text{ and } i\in \omega \}$ be the family in Proposition \ref{proposition.refining.P.subgroup}.

Let ${\mathfrak U}$ be the set of all ultrafilters ${\mathcal U}$ such that there exists $\xi \in \kappa$ such
that ${\mathcal U}$-limit of $\{z_{\xi,i}:\, i \in \omega\}$ is an element of the group $\langle Z \rangle$.

As in the proof of Theorem \ref{theorem.freeabelian.mcc.m+1not}, ${\mathfrak U}$ has cardinality at most ${\mathfrak c}$.

We will now obtain inductively a subset $I_\alpha \subseteq {\mathfrak c}\times \omega$ for each $\alpha \in [\omega , {\mathfrak c}[$ and the desired example will be $\langle \{z_{\beta ,i}:\, (\beta,i) \in \bigcup_{\alpha <{\mathfrak c} } I_\alpha \}\rangle$.

The sets $I_\alpha$ will satisfy the following conditions:

$1)$ $|I_\alpha|\leq |\alpha|+ \kappa$ for each $ \alpha < {\mathfrak c}$ and $\{I_\alpha:\, \alpha <{\mathfrak c}\}$ is a $\subseteq$-increasing chain in ${\mathcal P}({\mathfrak c}\times \omega)$;

$2)$ ${\mathfrak U}_\alpha $ is the set of all ultrafilters ${\mathcal U} \in {\mathfrak U}$ for which
there exists a nonempty finite support function $r:\, \kappa \longrightarrow {\mathbb Z}_P$  such that ${\mathcal U}$-lim $ (\sum_{\xi\in \supp r} r(\xi).z_{\xi,i}:\, i\in \omega )$ is an element of $\langle \{ z_{\beta , i}:\,  (\beta,i) \in I_\alpha \}\rangle$;

$3)$ $|{\mathfrak U}_\alpha|\leq |\alpha|+\kappa$ for each $ \alpha <{\mathfrak c}$ and $\{ {\mathfrak U}_\alpha:\, \alpha <{\mathfrak c}\}$ is $\subseteq$-increasing chain;

$4)$ $J_\alpha$ is a subset of ${\mathfrak c}\times \omega$ and for each
${\mathcal U}\in {\mathfrak U}_{\alpha}$, there exists $\xi <\kappa$ such that ${\mathcal U}$-limit of
$(z_{\xi,i}:\, i \in \omega ) \notin \langle \{z_{\beta ,i }:\, (\beta,i) \in {\mathfrak c} \setminus J_\alpha \} \rangle$;

$5)$ $|J_\alpha| \leq |\alpha|$ for each $\kappa \leq \alpha < {\mathfrak c}$ and $\{ J_\alpha :\, \alpha <{\mathfrak c}\}$ is a $\subseteq$-increasing chain in ${\mathcal P}({\mathfrak c}\times \omega)$;

$6)$ $\theta_\alpha \in [\omega, {\mathfrak c}[$ is the least ordinal $\theta \in {\mathfrak c}\setminus \{ \theta_\beta:\, \beta < \alpha \}$ such that $\bigcup_{n\in \omega \text{ and } i <\lambda_\theta} \text{ supp } g_{\theta , i}(n) \subseteq \bigcup_{\mu <\alpha }I_\mu$, for each $\alpha \in [\kappa , {\mathfrak c}[$;

$7)$ $\rho_\alpha \in [\omega, {\mathfrak c}[$ is such that $g_{\theta_\alpha}=g_{\rho_\alpha}$ for each $\alpha \in [\kappa , {\mathfrak c}[$;

$8)$ $\{\rho_\alpha\} \times \omega  \subseteq {\mathfrak c} \setminus ( \bigcup_{\mu <\alpha} (J_\mu \cup I_\mu))$, for each $\alpha \in [\kappa , {\mathfrak c}[$;

$9)$ $\{\rho_\alpha\} \times  \lambda_{\rho_\alpha} \subseteq I_\alpha$, for each $\alpha \in [\kappa , {\mathfrak c}[$ and

$10)$ $J_\alpha \cap I_\alpha=\emptyset$, for each $\alpha <{\mathfrak c}$.

 Once $I_\alpha$, $J_\alpha$, $\theta_\alpha$, $\rho_\alpha$, ${\mathfrak U}_\alpha$ are constructed for $\kappa \leq \alpha <{\mathfrak c} $ satisfying the conditions above, then $G=\langle \{ z_{\beta, i}:\, i\in \omega \text{ and } \beta \in \bigcup_{\alpha <{\mathfrak c}}I_\alpha\}\rangle$ is as required.

Suppose that the inductive construction is complete. The argument to show that the powers smaller than $\kappa$ are countably compact and that the $\kappa$-th power is not are as in the proof of Theorem \ref{theorem.freeabelian.mcc.m+1not}.

We will now proceed with the inductive construction:

{\bf Case 1} $[\alpha < \kappa]$. Set $I_\alpha=\kappa\times \omega$, for each $\alpha<\kappa$. 

The same argument used in the proof of Theorem \ref{theorem.freeabelian.mcc.m+1not} works if we make changes in notation and we obtain
${\mathfrak U}_\alpha=J_\alpha=\emptyset$, for each $\alpha <\kappa$.

{\bf Case 2}  $[\alpha \geq \kappa]$.

Basically, the same proof works. The only observation:

  Given $z_{{\mathcal U},n}$  the $ {\mathcal U}$-limit of $\{ z_{\xi,i}:\, i \in \omega \} $, for each $ \xi <\kappa$. As before, the cardinality of the set is bigger than the cardinality of the set added at the stage, thus, there is $n_{\mathcal U} <\kappa$ such that
$z_{{\mathcal U},n_{\mathcal U}} \notin \langle \{ z_{\beta,i}:\, (\beta,i) \in I_\alpha \}\rangle $, for each ${\mathcal U}\in {\mathfrak U}_\alpha \setminus \bigcup_{ \mu <\alpha}{\mathfrak U}_\mu$.

\smallskip 

 {\bf $3)$  The Continuum Hypothesis holds and $\kappa= \omega_1$}.
 
  We use $\omega \subseteq L \subseteq \omega_1$ such  that $L$ and $\omega_1\setminus L$ are unbounded in $\omega_1$.
 
 Let $Z=\{ z_{\alpha,i}:\, \alpha <{\mathfrak c} \text{ and } i\in \omega \}$ be the family in Proposition \ref{proposition.refining.P.subgroup}.
 
 Let ${\mathfrak U}$ be the set of all ultrafilters ${\mathcal U}$ such that for each $\xi \in \kappa$ such
 that ${\mathcal U}$-limit of $\{z_{\xi,i}:\, i \in \omega\}$ is an element of the group $\langle Z \rangle$.
 
 As in the proof of Theorem \ref{theorem.freeabelian.mcc.m+1not}, ${\mathfrak U}$ has cardinality at most ${\mathfrak c}$.
 
 We will now obtain inductively a subset $I_\alpha \subseteq {\mathfrak c}\times \omega$ for each $\alpha \in [\omega , {\mathfrak c}[$ and the desired example will be $\langle \{z_{\beta ,i}:\, (\beta,i) \in \bigcup_{\alpha <{\mathfrak c} } I_\alpha \cup (L\times \omega)\}\rangle$. 
 
 The sets $I_\alpha$ will satisfy the following conditions:
 
 $1)$ $|I_\alpha| <\omega_1={\mathfrak c}$, for each $ \alpha < \omega_1$, $I_\alpha \supset (L\cap \alpha)\times \omega$ and $\{I_\alpha:\, \alpha <\omega_1\}$ is a $\subseteq$-increasing chain in ${\mathcal P}({\mathfrak c}\times \omega)$;
 
 $2)$ ${\mathfrak U}_\alpha $ is the set of all ultrafilters ${\mathcal U} \in {\mathfrak U}$ for which
 there exists $r:\, L\cap \{\beta:\, \{\beta\} \times \omega \cap I_\alpha \neq \emptyset \} \longrightarrow {\mathbb Z}_P$ with nonempty finite support such that ${\mathcal U}$-lim $ (\sum_{\xi\in F} r(\xi).z_{\xi,i}:\, i\in \omega )$ is an element of $\langle \{ z_{\beta , i}:\,  (\beta,i) \in I_\alpha \}\rangle$;
 
 $3)$ $|{\mathfrak U}_\alpha| <\omega_1 $ for each $ \alpha <\omega_1$ and $\{ {\mathfrak U}_\alpha:\, \alpha <{\mathfrak c}\}$ is $\subseteq$-increasing chain;

 $4)$ $J_\alpha$ is a subset of ${\mathfrak c}\times \omega$ and for each
 ${\mathcal U}\in {\mathfrak U}_{\alpha}$, there exists $\xi \in L$ such that ${\mathcal U}$-limit of
 $(z_{\xi,i}:\, i \in \omega ) \notin \langle \{z_{\beta ,i }:\, (\beta,i) \in {\mathfrak c} \setminus J_\alpha \} \rangle$;
 
 $5)$ $|J_\alpha| \leq |\alpha|$ for each $\omega \leq\alpha < {\mathfrak c}$ and $\{ J_\alpha :\, \alpha <{\mathfrak c}\}$ is a $\subseteq$-increasing chain in ${\mathcal P}({\mathfrak c}\times \omega)$;
 
 $6)$ $\theta_\alpha \in [\omega, {\mathfrak c}[$ is the least ordinal $\theta \in {\mathfrak c}\setminus \{ \theta_\beta:\, \beta < \alpha \}$ such that $\bigcup_{n\in \omega \text{ and } i <\lambda_\theta} \text{ supp } g_{\theta , i}(n) \subseteq \bigcup_{\mu <\alpha }I_\mu$, for each $\alpha \in [\omega , {\mathfrak c}[$;
 
 $7)$ $\rho_\alpha \in [\omega, {\mathfrak c}[$ is such that $g_{\theta_\alpha}=g_{\rho_\alpha}$, whenever $\omega \leq \alpha <{\mathfrak c}$;
 
 $8)$ $\{\rho_\alpha\} \times \omega  \subseteq {\mathfrak c} \setminus ( \bigcup_{\mu <\alpha} (J_\mu \cup I_\mu))$, for each $\omega \leq \alpha < {\mathfrak c}$;
 
 $9)$ $\{\rho_\alpha\} \times \lambda_{\rho_\alpha} \subseteq I_\alpha$, for each $\omega \leq \alpha <{\mathfrak c}$ and
 
 $10)$ $J_\alpha \cap (L\cup I_\alpha)=\emptyset$, for each $\alpha <{\mathfrak c}$.
 
 \smallskip
 Before proving the conditions $1)-10)$, we note that the group 
   $G$ defined as $\langle \{z_{\beta ,i}:\, (\beta,i) \in \bigcup_{\alpha <{\mathfrak c} } I_\alpha \cup (L\times \omega)\}\rangle$ is as required. Using similar arguments as before,  one can show that $G^\omega$ is countably compact and $G^{\omega_1}$ is not countably compact. 
   
   \smallskip
  
  {\bf Case 1}. $I_\alpha = (\omega \times \omega) \cup (L \cap \omega )\times \omega=(L\cap \omega) \times \omega$. Then ${\mathfrak U}_\alpha=J_\alpha =\emptyset$, for each $\alpha <\omega$. The argument is similar as before.
  
  \smallskip
  
  Assume $\alpha\geq \omega$ and that the induction has been carried out  for $\mu <\alpha$.
  \smallskip
  
  {\bf Case 2}. The definition of $\theta_\alpha$ and $\rho_\alpha$ are as before and conditions $5)-8)$ are satisfied.
  
  Fix $\nu_\alpha \in L$ such that $\nu_\alpha
  > \rho_\alpha$ and $\nu_\alpha > \beta$ for each $(\beta, i) \in \bigcup_{\mu <\alpha}I_\mu \cup J_\mu$
  This will be used to define $J_\alpha$ later.

   Define $I_\alpha=\bigcup_{\mu <\alpha}I_\mu \cup (\{\rho_\alpha \} \times \lambda_{\rho_\alpha}) \cup (\{\nu_\alpha\} \times \omega)$. Then condition $1)$ and $9)$ are satisfied. 
   
   Define ${\mathfrak U}_\alpha$ as in condition $2)$. Similarly as before, condition $3)$ is satisfied.
   
   The definition of $J_\alpha$ is different and uses $\nu_\alpha$. Fix ${\mathcal U}\in {\mathfrak U}_\alpha$. Let $z_{\mathcal U}$ be the ${\mathcal U}$-limit of $(z_{\nu_\alpha, i}):\, i \in \omega)$. Let $r \in {\mathbb Z}^{({\mathfrak c}\times \omega)}$ be the such that $z=z_r$. 
   
   {\em Claim 1}. The support of $r$ is not a subset of $ \nu_\alpha \times \omega$. If that is the case, fix $a\in ({\mathbb Z}^P)^\omega$ of order $P$ and define  $D:\, L\times \omega \longrightarrow ({\mathbb Z}^P)^\omega$ such that $D(\beta,i)=a$ if $(\beta,i) \in ( \{\nu_\alpha\}\times \omega) \setminus \supp r$  and $D(\beta, i)=0$ otherwise. Let $\mu$ be such that $z_{\beta,i(\mu)}=D(\beta,i)$, for each $(\beta,i)\in L\times \omega$.
 
 By the fact that $z_{\beta,i}(\mu)=0$ for each $(\beta,i)\in (L\times \omega)\cup (\bigcup_{ \beta \in \nu_\alpha \setminus L } \{\beta\} \times (\omega \setminus \lambda_\beta)$, it follows by induction that $z_{\beta,i}(\mu)$ is an accumulation point for a sequence of $0$'s. Thus, $z_{\beta,i}(\mu)=0$, for each  $\beta \in \nu_\alpha \setminus L$ and $i<\lambda_\beta $. Hence, $z_{\beta,i}(\mu)=0$ for each $(\beta,i)\in (\nu_\alpha \times \omega  ) \cup \supp r $ and
 $z_{\beta,i}(\mu)=a$ for each $(\beta,i)\in (\{\nu_\alpha\} \times \omega  ) \setminus  \supp r $.
 
 Therefore, the $\supp r \setminus \nu_\alpha \times \omega \neq \emptyset$. Let $\beta_*$ be the largest ordinal for which $F_*=\supp r \cap (\{\beta_*\}\times \omega)\neq \emptyset$. 
 
 {\em Claim 2}. $ \beta_* \not\in L$
 
Suppose by contradiction that $\beta_* \in L$. Fix an independent set $(a_{\beta_*,i}:\, (\beta_*,i) \in F_*)$ contained in $ ({\mathbb Z}^P)^\omega$ . Define $E:\, L\times \omega \longrightarrow ({\mathbb Z}^P)^\omega$ such that $ E(\beta,i)=a_{\beta,i}$ if $(\beta,i) \in F_*$ and $E(\beta,i)=0$ otherwise. Let $\mu_*$ be the coordinate for which $z_{\beta,i}(\mu_*)=E(\beta,i)$, for each $(\beta,i)\in L\times \omega$.
Then, the ${\mathcal U}$-limit of $(z_{\nu_\alpha,i}:\, i \in \omega)=0$ and $z_r(\mu_*)= z_{r|_{F_*}}(\mu_*)=\sum_{(\beta_*,i)\in F_*}r(\beta_*,i)a_{\beta_*,i}\neq 0$, which is a contradiction.

It follow from Claim 1 and 2 that there exist $(\beta_{\mathcal U}, i_{\mathcal U}) \in \supp r \setminus (L\times \omega)$ such that $\beta_{\mathcal U}> \nu_\alpha$.

Set $J_\alpha = \bigcup_{\mu<\alpha}J_\mu \cup \{(\beta_{\mathcal U},i_{\mathcal U}):\, {\mathcal U}\in {\mathfrak U}_\alpha\}$. The points $(\beta_{\mathcal U},i_{\mathcal U})$ show that condition $4)$ is satisfied and the choice of $\nu_\alpha$ and $\beta_{\mathcal U}>\nu_\alpha$ imply that $10)$ is satisfied.

Thus all conditions $1)-10)$ are satified.
\end{proof}

\begin{example} In the Random model, for each cardinal $\kappa \leq\omega_1$, there exists a topological group topology on the Boolean group $G$ of cardinality ${\mathfrak c}$ such that $G^\gamma$ is countably compact for each $\gamma<\kappa$ and
$G^\kappa$ is not countably compact.
\end{example}

\begin{proof} Szeptycki and Tomita \cite{SzTo?} showed that the Boolean group $H$ of cardinality ${\mathfrak c}$ admits a group topology without non-trivial convergent sequences such that $H^\omega$ is countably compact. Applying
Theorem \ref{theorem.P.omegacc.omega1not.notch} we obtain the desired example.
\end{proof}

As earlier mentioned, the example of Hru\v sak, van Mill, Ramos and Shelah give:

\begin{example} In ZFC, for each cardinal $\kappa \leq\omega_1$, there exists a topological group topology on the Boolean group $G$ of cardinality ${\mathfrak c}$ such that $G^\gamma$ is countably compact for each $\gamma<\kappa$ and
	$G^\kappa$ is not countably compact.
\end{example}

\section[*]{Questions and Remarks}

This work has been inspired by questions on  surveys of Professor Comfort. We list some variations of his questions that remain open.

A natural question is the preservation of the algebraic structure of the example. The following questions are in this direction:

\begin{question} Let $G$ be a torsion Abelian group of cardinality ${\mathfrak c}$ and $\kappa$ some cardinal not greater than $ \omega_1$. Suppose that  $G$ is a topological group without non-trivial convergent sequences such that $G^\alpha$ is countably compact, for each $\alpha <\kappa$ with a group topology $\tau$. Is there a group topology $\sigma$ on $G$ which in addition makes its $\kappa$-th power not countably compact?
\end{question}

Under Martin's Axiom, an Abelian group of cardinality ${\mathfrak c}$ admits a countably compact group topology if and only if it admits a countably compact group topology without non-trivial convergent sequences.

\begin{question} Assume Martin's Axiom and let $G$ be an Abelian group of cardinality ${\mathfrak c}$.
 Does it hold that $G$ has a countably compact group topology  if and only if for each $\kappa \leq {\omega}$ there is a group topology on $G$ such that $\kappa$ is the least cardinal for which $G^\kappa$ fail to be countably compact? 
\end{question}

We do not know a model without selective ultrafilters for which there exists a countably compact free Abelian group.

\begin{question} Is there a model without selective ultrafilters in which there exists a group topology without non-trivial convergent sequences in some torsion-free Abelian group whose finite powers are countably compact?
\end{question}

In \cite{BoCaTo15} the authors showed that under Martin's Axiom there exists a topology without non-trivial convergent sequences in the real line that makes its square countably compact, but not its cube:

\begin{question}
	Is there a topological group without non-trivial convergent sequences in a torsion-free Abelian group such that every power is countably compact? In particular, the algebraic group ${\mathbb R}$ or ${\mathbb T}$ has such group topology?
\end{question}

After the result of Hru\v sak, van Mill, Ramos and Shelah, some new questions arise in ZFC. The author would like to thank Professor Hru\v sak for explaining the main idea of their construction before the completion of their preprint.

\begin{question}
$a)$	Which Abelian groups $G$  admit in ZFC a countably compact group topology without non-trivial convergent sequences? 
In particular, is there a non torsion Abelian group without non-trivial convergent sequences that is countably compact?

$b)$ Classify, in ZFC, the Abelian groups of size ${\mathfrak c}$ admitting a countably compact group topology.

$c)$ Is there in ZFC a free ultrafilter $p$ and a $p$-compact group without non-trivial convergent sequences? Is there $p$ for which such topolgoy does not exist?

$d)$ Is there a Wallace semigroup in ZFC? (A countably compact both-sided cancellative topological semigroup that is not algebraically a group). 
\end{question}

We note that a positive solution for the second question in  $a)$ yields a positive solution for $d)$.

\smallskip

The author would like to thank the referee for comments that  improved the presentation of this work.

\smallskip

I first met Professor Comfort as the external examiner for my doctoral thesis and in all few occasions I saw him in conferences or asked him something by e-mail, I felt his kindness and good spirits and will always remember him for this.

\end{document}